\documentclass[11pt]{amsart}
\usepackage[normalem]{ulem}
\usepackage{amssymb, adjustbox, enumerate, amsbsy, stmaryrd}
\usepackage{geometry}
\usepackage{todonotes}
\usepackage{amsfonts, amssymb, amscd}
\numberwithin{equation}{section}

\usepackage[symbol]{footmisc}

\usepackage{bm}
\usepackage{verbatim}
\usepackage{mathrsfs}
\usepackage{graphicx}
\usepackage{tikz-cd}
\usepackage{subcaption}
\usepackage{listings}
\usepackage{subfiles}
\usepackage[toc,page]{appendix}
\usepackage{mathtools}
\usepackage{comment}
\usepackage{enumerate}
\usepackage{enumitem}
\usepackage[all]{xy}

\usepackage{graphicx}
\graphicspath{{images/}}

\usepackage{appendix}
\usepackage{hyperref}
\hypersetup{
    colorlinks=true,
    citecolor=red,
    linkcolor=blue,
    filecolor=magenta,      
    urlcolor=red,
}
\def\bbeta{\boldsymbol{\beta}}
\def\ggamma{\boldsymbol{\gamma}}

\def\ddelta{\boldsymbol{\delta}}

\def\ideal#1.{I_{#1}}
\def\ring#1.{\mathcal {O}_{#1}}
\def\fring#1.{\hat{\mathcal {O}}_{#1}}
\def\proj#1.{\mathbb P(#1)}
\def\pr #1.{\mathbb P^{#1}}
\def\af #1.{\mathbb A^{#1}}
\def\Hz #1.{\mathbb F_{#1}}
\def\Hbz #1.{\overline{\mathbb F}_{#1}}
\def\pic#1.{\operatorname {Pic}\,(#1)}
\def\pico#1.{\operatorname{Pic}^0(#1)}
\def\picg#1.{\operatorname {Pic}^G(#1)}
\def\ner#1.{NS (#1)}
\def\rdown#1.{\llcorner#1\lrcorner}
\def\rup#1.{\ulcorner#1\urcorner}
\def\cone#1.{\operatorname {NE}(#1)}

\def\ccone#1.{\overline{\operatorname {NE}}(#1)}
\def\coef#1.{\frac{(#1-1)}{#1}}
\def\vit#1.{D_{\langle #1 \rangle}}
\def\mm#1.{\overline {M}_{0,#1}}
\def\H1#1.{H^1(#1,{\ring #1.})}
\def\ac#1.{\overline {\mathbb F}_{#1}}

\def\adj#1.{\frac {#1-1}{#1}}
\def\spn#1.{\overline{#1}}
\def\ses#1.#2.#3.{0\to #1\to #2\to #3 \to 0}
\def\pek#1.#2.{\Cal P^{#1}(#2)}
\def\plk#1.#2.{\Cal P^{\leq #1}(#2)}
\def\ev#1.{\operatorname{ev_{#1}}}
\def\bminv#1.{(\nu_1,s_1;\nu_2,s_2;\dots ;\nu_{#1},s_{#1};\nu_{r+1})}
\def\zinv#1.{(\nu_1,s_1;\nu_2,s_2;\dots ;\nu_{#1},s_{#1};0)}
\def\iinv#1.{(\nu_1,s_1;\nu_2,s_2;\dots ;\nu_{#1},s_{#1};\infty)}
\def\map#1.#2.{#1 \longrightarrow #2}
\def\rmap#1.#2.{#1 \dasharrow #2}
\def\emb#1.#2.{#1 \hookrightarrow #2}

\def\dim{\operatorname{dim}}

\def\e{\Cal E}

\def\e1{E_1}
\def\e2{E_2}

\def\OO{\mathcal O}

\newcommand{\po}{\ar@{}[dr]|{\text{\pigpenfont R}}}
\newcommand{\pb}{\ar@{}[dr]|{\text{\pigpenfont J}}}

\newcommand\Q{{\mathbb{Q}}}
\newcommand\R{{\mathbb{R}}}

\theoremstyle{definition}

\theoremstyle{definition}

\newtheorem{theorem}{Theorem}[section]
\newtheorem{lemma}[theorem]{Lemma}
\newtheorem{proposition}[theorem]{Proposition}
\newtheorem{corollary}[theorem]{Corollary}

\theoremstyle{definition}
\newtheorem{definition}[theorem]{Definition}

\newtheorem{remark}[theorem]{Remark}
\newtheorem{conjecture}[theorem]{Conjecture}

\begin{document} 

\author{Christopher Hacon} \address{Department of Mathematics\\ University of Utah\\ 155 S 1400 E\\ Salt Lake City, Utah 84112} \email{hacon@math.utah.edu}
\thanks{Christopher Hacon was partially supported by the DMS-2301374 and
by a grant from the Simons Foundation SFI-MPS-MOV-00006719-07.}
\author{Yi Li}
  \address{Department of Mathematics, Wuhan University, Hubei, Wuhan, China}
    \email{yilimath@whu.edu.cn}
\author{Lingyao Xie} 
\address{School of Mathematical Sciences, Zhejiang University, Hangzhou, China}
\email{xielingyao@zju.edu.cn}
%\address{Department of Mathematics\\ University of California, San Diego \\ 9500 Gilman Drive 0112\\ La Jolla, CA 92093-0112, USA} \email{l6xie@ucsd.edu}
\date{\today}

\title{Fujiki Class $\mathcal C$ Varieties and a K\"ahler Criterion}
\maketitle
\begin{abstract}
In this article, we show that if $(X,B+\bbeta)$ is a strongly $\mathbb{Q}$-factorial compact K\"ahler generalized klt pair with $B+\bbeta_X$ big, then the outcome of a flip or a divisorial contraction is again K\"ahler, and we give a criterion for
varieties in Fujiki's class $\mathcal C$ to be K\"ahler. We also prove the
existence of small $\mathbb Q$-factorializations for generalized klt pairs
and of dlt modifications for generalized pairs.
\end{abstract}
\tableofcontents
\section{Introduction}
The  classification of compact K\"ahler manifolds up to bimeromorphism is one of the central problems in analytic geometry with many important consequences and applications. 
For projective varieties, the situation is now largely understood thanks to the minimal model program and in particular to the results of \cite{BCHM10}. It is expected that analogues of the main results of \cite{BCHM10} should hold for compact K\"ahler varieties.

\begin{conjecture}
Let $(X,B+\bbeta)$ be a  generalized klt pair, where $X$ is compact K\"ahler and $B+\bbeta_X$ is modified big. Then:
\begin{enumerate}
\item if $K_X+B+\bbeta_X$ is pseudo-effective, then $(X,B+\bbeta)$ has a good minimal model;
\item if $K_X+B+\bbeta_X$ is not pseudo-effective, then $(X,B+\bbeta)$ has a Mori fiber space.
\end{enumerate}
\end{conjecture}

Note that the above conjecture is known in full generality for projective varieties \cite{DH24}.  We will address this conjecture in the forthcoming paper \cite{HX26}.
It is expected that the desired output is obtained by a finite sequence of flips and divisorial contractions, ending either with a good minimal model or with a Mori fiber space. 
In order to prove the above conjecture, we hope to imitate the approach of \cite{BCHM10}. The main technical issues that must be settled are the cone theorem  (cf. \cite{HP24}),
the contraction theorem for generalized klt pairs, and the existence of flips for generalized klt pairs (cf. \cite{DH23,DH25}).
In this paper we resolve 
another important issue: we show that flips and divisorial contractions preserve the K\"ahler condition.

\begin{theorem}\label{t-contK}
Let $(X,B+\bbeta)$ be a generalized klt pair, where $X$ is compact, K\"ahler,  strongly
$\Q$-factorial. If
$f: X\to Z$ is a flipping or divisorial contraction associated to an extremal ray $\Gamma$, then $Z$ is K\"ahler.
\end{theorem}

This result is essential for the K\"ahler minimal model program, to guarantee that after each divisorial contraction or flip we remain within the category of compact K\"ahler spaces, so that we can continue the minimal model program. In dimension three, this was proved in \cite{HP16,CHP16,DHY23} using a variant of Kleiman's K\"ahlerness criterion, which does not hold in higher dimensions. In this paper we show that a version of Kleiman's K\"ahlerness criterion (Theorem \ref{t-ray}) holds 
 for a big and nef adjoint class $\alpha = K_X + B + \bbeta_X$ on a compact complex variety in Fujiki's class $\mathcal{C}$ with generalized klt singularities. Thus, $\alpha$ is K\"ahler if and only if there are no $\alpha$-trivial curves. We also expect that a result similar to Theorem \ref{t-contK} holds without
assuming that $\alpha$ is big, so that if $f:X\to Y$ is a Mori fiber space, then $Y$ is K\"ahler. We do not pursue this here, however see \cite[Theorem 7.1]{HX26}.

Our second main result gives a criterion for a variety in Fujiki's class $\mathcal C$ to be K\"ahler.

\begin{theorem}\label{c-1} Let $(X,B+\bbeta)$ be a generalized klt pair, where $X$ is compact in Fujiki's class $\mathcal C$. If $X$ is not K\"ahler, then either
\begin{enumerate}
\item $X$ contains a rational curve $C$ such that $-[C]\in \overline{\rm NA}(X)$, or
\item there is a small $\Q$-factorial K\"ahler modification $\mu:X^{\rm qf}\to X$. 
\end{enumerate}
\end{theorem}

An analogous statement for projective varieties was proved in \cite{VP21}. Our proof follows the ideas of \cite{VP21}, applying results from the K\"ahler minimal model program, especially the results of \cite{Fuj22}. %The main difference between our proof and \cite{VP21} is that we use the finiteness of log canonical models, \cite[Theorem E]{Fuj22}, instead of running a relative minimal model program and carefully analyzing the steps of the MMP.
As an immediate consequence of Theorem \ref{c-1}, we obtain the following K\"ahlerness criterion.

\begin{corollary}\label{c-Kahlercritera}
Let $(X,B+\bbeta)$ be a generalized klt pair, where $X$ is strongly $\Q$-factorial, compact, in Fujiki's class
$\mathcal C$. Then $X$ is K\"ahler if and only if $\overline{\rm NA}(X)$ does not
contain a class of the form $-[C]$, where $C$ is a rational curve.
\end{corollary}

Along the way we establish several technical results that are of independent interest and are used in the proofs of the theorems above. The first concerns the existence of small $\Q$-factorializations for generalized klt pairs.

\begin{theorem}[{Small $\Q$-factorial modifications, Theorem \ref{t-3}}]\label{t-smallQfactorial}
Let $(X,B+\bbeta)$ be a compact generalized klt pair. Then there exists
a small projective bimeromorphic morphism $\nu: X'\to X$ such that $X'$ is normal and strongly $\Q$-factorial.
\end{theorem}
This result is shown for K\"ahler threefolds with generalized klt singularities, and in higher dimensions locally over a relative compact Stein open neighborhood of $X$ in \cite{DHY23}.

The next theorem establishes the existence of global dlt modifications for generalized pairs.

\begin{theorem}[{Global dlt models, Theorem \ref{t-dltmodel}}]\label{t-dltmodel-i}Let $X$  be a compact analytic variety and  $(X,B+\bbeta )$ a generalized pair. Then there exists a birational projective morphism $f^{\rm m}:X^{\rm m}\to X$ such that $X^{\rm m}$ is strongly $\Q$-factorial, all $f^{\rm m}$-exceptional divisors $P$ have discrepancy $a(P,X,B+\bbeta)\leq -1$ and
 $(X^{\rm m}, B^{\rm m})$ is dlt where $B^{\rm m}={(f^{\rm m})}^{-1}_*(B\wedge {\rm Supp}(B))+{\rm Ex}(f^{\rm m})$.
\end{theorem}
Following the ideas of \cite{Fil20}, the existence of dlt models was established
in \cite{HP24} under the additional assumption that $X$ is relatively compact
and Stein. We observe that a similar proof applies, with minor changes, once the relative
minimal model program is available (see Theorem~\ref{t-projcone}).

Finally, we establish some results on the relative minimal model program for Moishezon contraction morphisms. %which are not necessarily projective. 
Although these results are used in our paper as an auxiliary ingredient, we present them here as they may be useful in other contexts.

\begin{proposition}[{Proposition \ref{p-relmmp}} and Theorem \ref{t-3}]
Let $f:X\to Y$ be a contraction morphism of normal compact complex analytic varieties. Assume that $(X,B+\bbeta )$ is a generalized klt pair such that $B+\bbeta_X$ is modified big over $Y$ and $\bbeta \equiv _Y \mathbf N$ where $\mathbf N$ is an $\R$-Cartier b-divisor. Then
\begin{enumerate}\item If $K_X+B+\bbeta_X$ is pseudo-effective over $Y$, then $(X,B+\bbeta)$ has a good minimal model over $Y$.
\item  If $K_X+B+\bbeta_X$ is not pseudo-effective over $Y$, then $(X,B+\bbeta)$ is birational to a Fano fibration  over $Y$.
\end{enumerate}
\end{proposition}

\begin{proposition}[{Proposition \ref{c-relmmp} and Theorem \ref{t-3}}] Let $f:X\to Y$ be a bimeromorphic morphism of normal compact K\"ahler varieties with rational singularities. Assume that $(X,B+\bbeta )$ is a strongly $\Q$-factorial generalized klt pair. %such that $B+\bbeta_X$ is big over $Y$. 
Then we can run a $(K_X+B+\bbeta_X)$-relative minimal model program over $Y$ which ends with a good minimal model over $Y$.
\end{proposition}

\subsection{Strategy of the proof}
Let us briefly sketch the idea of the proof and give an outline of this paper. We begin with the proof of Theorem \ref{t-contK}. In the projective setting, if $(X,B)$ is a projective klt pair and $f: X \to Y$ is a flipping or divisorial contraction, then the projectivity of $Y$ follows from the base point free theorem. Indeed, the supporting divisor $K_X+B+H$ of the contraction (for some ample divisor $H$ on $X$) is semi-ample, and it is the pullback of an ample divisor on $Y$. Hence $Y$ is projective. In the K\"ahler setting the transcendental base point free theorem was not available at the time of writing this paper (see \cite[Conjecture~1.6]{DH25}), so a different argument was required. Note that the transcendental base point free conjecture has recently been established in \cite{HX26}. %\sout{Our arguments do not rely on \cite{HX26}; however several arguments of \cite{HX26} rely on the results of this paper.}

Let $
\alpha=[K_X+B+\bbeta_X+\omega]$ 
be a supporting class of the contraction $f:X\to Y$, where $\omega$ is a K\"ahler form on $X$. We first claim that there exists a nef and big class $\alpha_Y$ on $Y$ such that $
\alpha=f^*\alpha_Y$. Note that this can be shown by arguing locally over the base $Y$ and that since $f$ is birational, it is Moishezon. Passing to a resolution, we may assume that $f$ is projective and then we apply the methods of \cite{Fuj22}. 
It remains to prove that $\alpha_Y$ is K\"ahler. For this purpose, we apply Theorem \ref{t-ray}, which serves as a substitute for Kleiman's criterion for K\"ahlerness used in the threefold setting in \cite{HP16,CHP16,DHY23}. This theorem is a variant of \cite[Proposition 3.1]{HP24} with the important difference that we do not assume $\Q$-factoriality and K\"ahlerness of the varieties. Roughly speaking, Theorem \ref{t-ray} says that if the nef and big class $\alpha_Y$ is not K\"ahler, then  some irreducible component of its null locus is covered by $\alpha_Y$-trivial  rational curves. The proof follows the strategy of \cite{HP24}: we consider a maximal dimensional irreducible
component $Z$ of the null locus of $\alpha_Y$, take a
smooth model $Z'$ of $Z$, and study the MRC fibration
$\varphi: Z'\dasharrow W$, whose general fiber we denote by $F$. We then show that $Z'$ is uniruled and that $\alpha_Y|_F$ is not big. This produces a fiber type contraction $F\to \bar F$, and hence $F$ is covered by $\alpha_Y$-trivial curves. The main difference between our approach and \cite{HP24} is that, a priori $Y$ is only Fujiki, there is no ambient K\"ahler class available. We therefore replace the intersection-theoretic argument of \cite{HP24} by an MMP
argument, 
so that \cite[Theorem 2.2]{HP24} applies and gives the uniruledness of $Z'$.

We next turn our attention to the proof of Theorem \ref{c-1}. Since $X$ belongs to Fujiki's class $\mathcal C$, we choose a resolution $
\nu: X'\to X$ such that $\nu$ is projective and $X'$ is K\"ahler.
After fixing a sufficiently general K\"ahler class $\omega'$ on $X'$, we consider the relative nef threshold $t$ of $
K_{X'}+B'+\bbeta_{X'}
$
over $X$ with respect to $\omega'$. If this threshold is zero, it follows that the resolution has no exceptional divisors, and we obtain a small morphism. Combining this with Theorem \ref{t-smallQfactorial} we obtain the small $\Q$-factorial K\"ahler modification, and we are in case (2) of the theorem. We may therefore assume that the relative nef threshold $t$ is positive. We then take the relative $(K_{X'}+B'+\bbeta_{X'}+t\omega')$-log canonical model over $X$. By the cone theorem, this either produces a rational curve $C\subset X$ such that $
-[C]\in \overline{\rm NA}(X)$, 
which is case (1) of the theorem, or it produces a divisorial or flipping contraction. In the latter case, Theorem \ref{t-contK} shows that the resulting space is again K\"ahler, so the procedure can be repeated. The finiteness of log canonical models \cite[Theorem E]{Fuj22} implies that this process terminates and gives the desired small $\Q$-factorial K\"ahler modification.

Finally, we outline the ideas behind the proofs of the small $\mathbb{Q}$-factorialization theorem and the dlt modification theorem. Locally, one can work on relatively compact Stein open subsets and apply the corresponding local results, namely \cite[Theorem 2.19]{DHY23} and \cite[Theorem 1.6]{HP24}. Thus the main difficulty is to globalize these local constructions. For small $\Q$-factorial modifications, we show that there exists a finite collection of global rank-one reflexive sheaves which locally generate all rank-one reflexive sheaves. This reduces the problem to constructing a small $\Q$-factorial modification for this finite set of generators, which follows from Proposition \ref{t-2}. For dlt modifications, we follow the strategy of \cite{Fil20}: we run a relative MMP over the log resolution, carefully controlling the places being contracted by the MMP so that every exceptional divisor with discrepancy $a>-1$ is contracted, while preserving the dlt condition. The main difference in our setting is that we need Theorem \ref{t-projcone} to run the relative MMP for analytic varieties.

\subsection{Organization of the paper}

This article is organized in the following manner. In Section 2, we collect preliminary results used throughout the paper. In Section 3, we prove the existence of small $\Q$-factorializations and dlt modifications, namely Theorems \ref{t-smallQfactorial} and \ref{t-dltmodel-i}. Section 4 is devoted to the proof of Theorem \ref{t-contK}, which shows that flips and divisorial contractions preserve the K\"ahler condition. In the final section, we apply these results to prove a K\"ahlerness criterion for varieties in Fujiki's class $\mathcal{C}$, Theorem \ref{c-1}.

\medskip

\noindent

\textbf{Acknowledgements.} The authors would like to thank Mihai P\u{a}un for useful comments and suggestions.

\textbf{AI statement.} AI was used to find mathematical inaccuracies and typos, but it was not used to generate any of the mathematical statements in this paper. (The example in Remark~2.16 arose from the second author's communication with an AI system. The proof of Corollary~\ref{c-Kahleralphatrivial} is inspired by the approach of \cite[Theorem 4.1]{DH24}; the second author used an AI system to help adapt the argument of \cite{DH24} to the present setting.)

\section{Preliminaries}

%In this section, we introduce some preparatory notions and results to be used in the rest part of the paper. 

A complex analytic variety $X$ is a \emph{K\"ahler variety} if there exists a K\"ahler
form $\omega$ on $X$, i.e., a positive closed real $(1,1)$-form
$\omega \in \mathcal{A}_{\mathbb{R}}^{1,1}(X)$ such that for every singular point
$x \in X$ there exists an open neighborhood $U \subseteq X$ of $x$, a closed embedding
$\iota_U : U \hookrightarrow V$ into an open subset $V \subseteq \mathbb{C}^N$,
and a strictly plurisubharmonic $C^{\infty}$ function $f : V \to \mathbb{R}$
satisfying $\left.\omega\right|_{U \cap X_{\mathrm{sm}}}
= \left.(i \partial \bar{\partial} f)\right|_{U \cap X_{\mathrm{sm}}}$.

A compact complex analytic variety $Y$ is said to be in \emph{Fujiki's class} $\mathcal C$ if there exists
a bimeromorphic modification $f : X \to Y$ such that $X$ is a compact K\"ahler
manifold (see \cite[Definition 2.2]{DH25} for more details). 

Given a normal complex analytic variety $X$, let $\mathcal{H}_X$ be the sheaf of real parts of holomorphic functions multiplied by $\sqrt{-1}$, and define the \textit{Bott--Chern cohomology} to be $H^{1,1}_{\rm{BC}}(X):=  H^1(X, \mathcal{H}_X)$. A normal complex analytic variety $X$ is said to be \textit{strongly $\mathbb{Q}$-factorial} if every divisorial sheaf is $\mathbb{Q}$-Cartier, i.e. for every $\mathcal{F} \in W(X)$ (where $W(X)$ denotes the group of divisorial sheaves) there exists $m \in \mathbb{N}^*$ such that the reflexive power $\mathcal{F}^{[m]}$ is locally free. 
 Given a contraction morphism $f: X\to Y$ between normal compact complex analytic varieties, we define the \textit{relative Bott--Chern Picard number} by $\rho_{\rm BC}(X/Y):=\dim H^{1,1}_{\rm BC}(X)- \dim H^{1,1}_{\rm BC}(Y)$.

\begin{lemma}\label{l-rel} Let $f:X\to Y$ be a proper bimeromorphic morphism of normal varieties  with rational singularities such that $Y$ is relatively compact and Stein. For any $\alpha\in H^{1,1}_{\rm BC}(X)$, there is an $\R$-line bundle $L\in H^1(X,\OO _X^*)\otimes \R$ such that $[L]=\alpha$.  
\end{lemma}
\begin{proof} Since $X,Y$ have rational singularities, then $R^if_*\OO _X=0$ for $i>0$.
Since $f_*\OO _X=\OO _Y$ and $f_*\mathcal H _X=\mathcal H_Y$ (by \cite[Lemma 2.5]{DH24}), pushing forward the short exact sequence \[0\to \R\to \OO _X\to \mathcal H _X\to 0,\]it follows that $R^1f_*\R=0$ and $d:R^1f_*\mathcal H _X\to R^2f_*\R$ is an isomorphism. 
Pushing forward the short exact sequence
\[0\to \mathbb Z\to \OO _X\to \OO ^*_X\to 0,\] one sees that $\delta: R^1f_*(\OO_X^*)\to R^2f_*\mathbb Z$ is surjective.
Therefore we may pick $L=\sum r_iL_i$ where $L_i \in R^1f_*(\OO_X^*)$ and $r_i\in \R$ such that $\delta (L):=\sum r_i\delta(L_i)=d(\alpha)$ i.e. $[L]=\alpha$.
\end{proof}

\subsection{Positivity of Bott--Chern classes} We next introduce some positivity notions that will be used throughout this paper. 
\begin{definition}
Let $X$ be a normal compact complex analytic variety, and let $\alpha\in H^{1,1}_{\rm{BC}}(X)$ be a $(1,1)$-class. 
\begin{enumerate}
\item The class $\alpha \in H^{1,1}_{\rm{BC}}(X)$
is called \emph{K\"ahler} if it contains a K\"ahler form $\omega$; in this case $X$ is itself a K\"ahler variety.
\item The class $\alpha \in H^{1,1}_{\rm{BC}}(X)$ is called \emph{nef} if for some
positive $(1,1)$-form $\omega$ on $X$, for every $\varepsilon > 0$ there exists
$f_{\varepsilon} \in \mathcal{A}^{0}(X)$ such that
$$
  \alpha + i\partial\bar{\partial} f_{\varepsilon} \geq -\varepsilon\,\omega.
$$
\item The class $\alpha \in H^{1,1}_{\rm{BC}}(X)$ is \emph{pseudo-effective} if it contains
a positive current $T \geq 0$.
\item The class $\alpha \in H^{1,1}_{\rm{BC}}(X)$ is
called \emph{big} if it contains a K\"ahler current $T$, that $T\ge \omega$ for some Hermitian metric $\omega$ on $X$. If $f:X\to Y$ is a birational morphism and $\alpha \in H^{1,1}_{\rm{BC}}(X)$ is big, then we say that $f_*\alpha$ is {\it modified big}. 
\item The class $\alpha \in H^{1,1}_{\rm{BC}}(X)$ is called \textit{semi-ample} if there exists a contraction morphism $f:X\to Z$ with a K\"ahler class $\omega_Z$ such that $\alpha = f^* \omega_Z$.  
\end{enumerate}
\end{definition}
\begin{remark}
When $X$ is a K\"ahler variety, it is easy to see that a class $\alpha \in H^{1,1}_{\rm{BC}}(X)$ 
is nef if and only if it lies in the closure of the K\"ahler cone (here K\"ahler cone means the non-empty open convex cone generated 
by all K\"ahler classes in $H^{1,1}_{\rm{BC}}(X)$).
\end{remark}
\begin{remark}
It is well known that $\alpha \in H^{1,1}_{\rm{BC}}(X)$ is big if and only if there is a K\"ahler current $T\in \alpha$. We also remark that if $\alpha$ is big and $f:X'\to X$ is a birational morphism, then $f^*\alpha$ is also big (cf. \cite[Proposition 2.30]{HLR26}).
\end{remark}

% \begin{lemma}
% Let $f:X\to Y$ be a proper bimeromorphic morphism. Let $G$ be an effective $f$-exceptional divisor, and $\alpha \in H^{1,1}_{\rm{BC}}(X)$. Then
% $$\text{vol}(f^* \alpha + G) = \text{vol}(\alpha).$$
% In particular, if $f^* \alpha + G$ is big then $\alpha$ is big.
% \end{lemma}
% \begin{proof}

% \end{proof}

\begin{lemma}\label{l-modnef}
Let $(X,B+\bbeta)$ be a generalized pair, and assume that $B+\bbeta_X$ is modified big. Then there exists a birational morphism $\nu :X'\to X$ and an effective $\nu$-exceptional divisor $E$ such that if $K_{X'}+B'+\bbeta _{X'}=\nu ^*(K_X+B+\bbeta _X)$, then $B'+\bbeta_{X'} + E$ is big and $B'^{\geq 0}+\bbeta_{X'} +\epsilon E$ is big for any $\epsilon >0$. 
\end{lemma}
\begin{proof} We may pick a resolution $\nu: X'\to X$ and a big class $\Theta$ such that $\nu _* \Theta =B+\bbeta _X$. Possibly replacing $X'$ by a higher model, we may also assume that $\bbeta$ descends to $X'$ and in particular $\bbeta _{X'}$ is nef. We write $K_{X'}+B'+\bbeta _{X'}=\nu ^*(K_X+B+\bbeta_X)$. Then $\nu _*(\Theta -B'-\bbeta _{X'})=0$ and so, by Demailly's second support theorem (\cite[Proposition (2.13)]{agbook}), we see that $ \Theta-(B'+\bbeta_{X'})$ is a $\nu$-exceptional divisor $F$. Decompose $F = F^{\ge 0}- F^{\le 0}$, where $F^{\ge 0}\ge 0$, $F^{\le 0}\ge 0$ and $F^{\ge 0}$ and $F^{\le 0}$ have no common components. Define $E:= F^{\ge 0}$. Then we claim that $B'+\bbeta_{X'}+E$ is big. Indeed, by definition
$$B^{\prime}+\bbeta_{X^{\prime}}+E \equiv \Theta+F^{\le 0},$$
and the right-hand side is big. Since ${B^{\prime}}^{\geq 0}+\bbeta_{X^{\prime}}$ is pseudo-effective, it follows easily that \[B'^{\geq 0}+\bbeta_{X'} +\epsilon E=(1-\epsilon)\big(B'^{\ge0}+\bbeta_{X'}\big)+\epsilon\big(B'^{\ge0}+\bbeta_{X'}+E\big)\] is big for any $\epsilon >0$. 
% By assumption, there exists a birational morphism $\mu: Y\to X$ and a big class $\Theta$ such that $\mu_* \Theta = B+\bbeta_X$. On the other hand, by definition of generalized pair, there exists a birational morphism $\nu:X'\to X$ such that $\bbeta$ descends to $X'$. We then take a common resolution, as shown in the diagram below 
% \begin{center}
% \begin{tikzcd}
% 	W & Y \\
% 	{X'} & X.
% 	\arrow["q", from=1-1, to=1-2]
% 	\arrow["p"', from=1-1, to=2-1]
% 	\arrow["\mu", from=1-2, to=2-2]
% 	\arrow["\nu"', from=2-1, to=2-2]
% \end{tikzcd}
% \end{center}
% Then $$\nu_*(p_*q^*\Theta)  = B+\bbeta_X= \nu_*(B'+\bbeta_{X'}).$$ Since $\nu_*(p_*q^* \Theta - (B' + \bbeta_{X'})) = 0$, by Demailly's second support theorem (\cite[Proposition (2.13)]{agbook}), we see that $p_*q^* \Theta-(B'+\bbeta_{X'})$ is a $\nu$-exceptional divisor $F$. Decompose $F = F^{\ge 0}- F^{\le 0}$, where $F^{\ge 0}\ge 0$, $F^{\le 0}\ge 0$ and $F^{\ge 0}$ and $F^{\le 0}$ have no common components. Define $E:= F^{\ge 0}$. Then we claim that $B'+\bbeta_{X'}+E$ is big. Indeed, by definition
% $$B^{\prime}+\bbeta_{X^{\prime}}+F^{\ge 0} \equiv p_* q^* \Theta+F^{\le 0},$$
% and the right-hand side is big. Hence so is the left-hand side.
% Since  $B'^{\geq 0}+\bbeta_{X'}$ is pseudo-effective, then $B'^{\ge0}+\beta_{X'}+\epsilon E=(1-\epsilon)\big(B'^{\ge0}+\beta_{X'}\big)+\epsilon\big(B'^{\ge0}+\beta_{X'}+E\big)=$ is big for any $\epsilon >0$. 
\end{proof}

We will also require some positivity notions in the relative setting. 

\begin{definition}
Let $f: X \rightarrow S$ be a proper surjective morphism from a normal complex analytic variety $X$ onto a relatively compact base $S$, and let $\alpha \in H_{\mathrm{BC}}^{1,1}(X)$ be a $(1,1)$-class. We say that:
\begin{enumerate}
\item The class $\alpha \in H^{1,1}_{\rm{BC}}(X)$ is \textit{$f$-nef} or \textit{relatively nef over $S$} if for any $s\in S$, the restriction $\alpha_s := \alpha|_{X_s}$ is a nef class on the fiber $X_s$. Equivalently, $\alpha \in H^{1,1}_{\rm{BC}}(X)$ is {$f$-nef} if there is a cover $S=\cup S_i$ by open subsets such that each $\alpha |_{f^{-1}(S_i)}$ is nef.
\item The class $\alpha \in H^{1,1}_{\rm{BC}}(X)$ is \textit{$f$-pseudo-effective over $S$} or \textit{relatively pseudo-effective over $S$} if there is a complement $V$ of a countable union of Zariski closed proper subsets of $S$ such that $\alpha_s:=\left.\alpha\right|_{X_s}$ is pseudo-effective for any $s \in V$.
\item The class $\alpha \in H_{\mathrm{BC}}^{1,1}(X)$ is \textit{$f$-big over $S$} or \textit{relatively big over $S$} if there is a dense Zariski open subset $U$ of $S$ such that $\alpha_s$ is big for any $s \in U$.

\item The class $\alpha \in H_{\mathrm{BC}}^{1,1}(X)$ is \textit{$f$-semi-ample over $S$} or \textit{relatively semi-ample over $S$} if there exists a contraction $\varphi: X\to Y/S$ and a relative K\"ahler class $\omega_Y$ on $Y$ over $S$ such that $\alpha = \varphi^* \omega_Y$. 
\end{enumerate}
\end{definition}

\begin{lemma}[{Negativity lemma for movable divisors, \cite{Fuj22}}]\label{l-negativity}
Let $\pi:X'\to X$ be a projective bimeromorphic morphism and let $E$ be an $\mathbb{R}$-Cartier divisor on $X'$ with $E\in\overline{\operatorname{Mov}}(X'/X;W)$, where $W\subset X$ is a compact subset and $\overline{\operatorname{Mov}}(X'/X;W)$ denotes the closure of the cone of relative movable divisors over $W$. If $-\pi_*E|_U$ is effective for some open neighborhood $U$ of $W$, then $-E|_{\pi^{-1}(U)}$ is effective.
\end{lemma}

\begin{proposition}\label{p-negZariski}
Let $f:X\to Y$ be a projective bimeromorphic morphism between compact complex varieties. Assume that $G\ge 0$ is an $f$-exceptional $\mathbb{R}$-Cartier divisor. Then the negative part of the divisorial Zariski decomposition satisfies $$N_{\sigma}(G/Y) = G \ge 0.$$
\end{proposition}
\begin{proof}
By definition, the positive part $P_\sigma(G/Y):= G - N_{\sigma}(G/Y)$ of the
divisorial Zariski decomposition is represented by a
$\mathbb{R}$-divisor. Since $G \ge 0$, $N_{\sigma}(G/Y) \le G$. Therefore $P = G - N_{\sigma}(G/Y) \ge 0$
is effective, and it is movable over $Y$ by
\cite[Lemma 4.10]{FKL16}. Since $G$ is $f$-exceptional and $\operatorname{Supp}(P) \subset
\operatorname{Supp}(G)$, the divisor $P$ is also $f$-exceptional.
Therefore, by the negativity lemma \ref{l-negativity}, $P = 0$, and hence
$N_{\sigma}(G/Y) = G$.
\end{proof}

\subsection{Generalized pairs and singularities}
Throughout this paper we work with \emph{generalized pairs} 
$(X,B+\bbeta )$, where $\bbeta$ is a b-nef b-$(1,1)$ current. By adding this extra flexibility we are able to treat transcendental problems within the framework of the minimal model
program.
\begin{definition}[{Generalized pairs, \cite[Definition 2.7]{DHY23}}]
Let $f: X \rightarrow S$ be a proper morphism of complex analytic varieties, where $S$ is relatively compact, $\nu: X^{\prime} \rightarrow X$ a resolution, $B^{\prime}$ an $\mathbb{R}$-divisor on $X^{\prime}$ with simple normal crossing support such that $B:=\nu_* B^{\prime} \geq 0$, and $\boldsymbol{\beta}$ a closed $\mathrm{b}$-$(1,1)$ current. We say that $(X/S, B+\boldsymbol{\beta})$ is a \textit{generalized pair} if

(1) $\boldsymbol{\beta}$ is a real closed smooth b-$(1,1)$ current that descends to $X^{\prime}$,

(2) $\left[\boldsymbol{\beta}_{X^{\prime}}\right] \in H_{\mathrm{BC}}^{1,1}\left(X^{\prime}\right)$ is nef over $S$, and

(3) $\left[K_{X^{\prime}}+B^{\prime}+\boldsymbol{\beta}_{X^{\prime}}\right]=\nu^* \gamma$ for some $\gamma \in H_{\mathrm{BC}}^{1,1}(X)$.
\end{definition}
Note that the trace of the b-$(1,1)$ current $\bbeta$ on $X$ may not have local potentials. But if $\bbeta_{X}$ has local potentials then it has plurisubharmonic local potentials (cf. \cite[Remark 2.6.(ii)]{DHY23}). If $\beta$ is a closed positive $(1,1)$-current on $X$ with local (psh) potentials, then we may define a b-$(1,1)$ current $\bar{\beta}$, such that for any bimeromorphic morphism $\nu: X^{\prime} \rightarrow X$ we let $\bar{\beta}_{X^{\prime}}:=\nu^* \beta$, which is well defined by \cite[Claim 2.5]{DHY23}. 

Let $(X{ /S}, B+\bbeta)$ be a generalized pair as above, and let $E$ be a prime divisor on some birational model of $X$. We define the \textit{generalized discrepancy} of $E$ with respect to the above generalized pair as follows. Let $\nu: X^{\prime} \rightarrow X$ be a log resolution, such that $B^{\prime}$ is an $\mathbb{R}$-divisor with simple normal crossing support, $\boldsymbol{\beta}$ is the positive b-$(1,1)$ current that descends to $X^{\prime}$ so that $\boldsymbol{\beta}_{X^{\prime}}$ is nef over $S$ and $\boldsymbol{\beta}=\overline{\boldsymbol{\beta}_{X^{\prime}}}$, and let $\left[K_{X^{\prime}}+B^{\prime}+\boldsymbol{\beta}_{X^{\prime}}\right]=\nu^* \gamma$ for some $\gamma \in H_{\mathrm{BC}}^{1,1}(X, \mathbb{R})$. We can then write
$$
K_{X^{\prime}}+B^{\prime}+\boldsymbol{\beta}_{X^{\prime}}=\nu^*\left(K_X+B+\boldsymbol{\beta}_X\right) .
$$
The generalized discrepancy $a(E,X{/S}, B+ \boldsymbol{\beta})$ of $E$ is defined to be $-\operatorname{coeff}_E\left(B^{\prime}\right)$.

A generalized pair $(X{/S}, B+\bbeta)$ is said to have \textit{generalized klt singularities} (resp. \textit{generalized lc singularities}) if the generalized discrepancy of any prime divisor $E$ over $X$ is $a(E, X{/S}, B+ \boldsymbol{\beta})>-1$ (resp. $a(E, X{/S}, B+ \boldsymbol{\beta}) \geq-1$ ).

A generalized pair $(X, B+\boldsymbol{\beta})$ is said to have \textit{generalized dlt singularities}  if there is an open subset $U \subset X$ such that $\left(U,\left.(B+\bbeta)\right|_U\right)$ is a log resolution (of itself) and $-1 \leq a(E, X{/S}, B+\bbeta) \leq 0$ for any prime divisor $E$ on $U$ and $-1<a(E, X{/S}, B+\bbeta) $ for any prime divisor $E$ over $X$ with center contained in $X \setminus U$.

\begin{remark}
It is straightforward to check that a generalized pair $(X, B + \bbeta)$ (where we assume that $S$ is a point) is generalized klt if and only if there exists an open cover $X = \bigcup_{i \in I} X_i$ such that $(X_i, B|_{X_i} + \bbeta|_{X_i})$ is generalized klt for every $i \in I$. For generalized dlt singularities, however, the condition is not analytically local, cf. \cite[Example~3.11]{Fuj22}.

\end{remark}
We can also define the log canonical threshold for generalized pairs. Let $(X{/S},B+ \boldsymbol{\beta})$ be a generalized pair, $\eta$ a locally $\partial\bar \partial$-exact (1,1)-form on $X$ such that for some resolution
$\nu:X'\to X$ we have $$\nu^* \eta = \eta ' + F$$ where
$[\eta']\in H^{1,1}_{\rm BC}(X')$ is nef over $S$ and $F\ge 0$ is an effective $\mathbb{R}$-divisor.   We write $\nu^*(K_X+B+ \bbeta_X) = K_{X'}+B'+\bbeta_{X'}$. By further blowing up $X'$, we may assume that $F+B'$ is SNC and $\bbeta$ descends to $X'$. Then we define the \textit{log canonical threshold}  $$\text{LCT}(X/S, B + \boldsymbol{ \beta};\overline{\eta'}+\nu_*F;Z) = \sup\{t >0 \mid (X' { / S}, B' + t F) \text{ is log canonical near  }Z\}.$$

\begin{lemma}\label{l-gklt}
Let $f: X'\to X$ be a proper bimeromorphic morphism over a relatively compact Stein open subset $X$. Assume that $(X,B+\bbeta)$ is generalized klt and $K_{X'}+B_{X'}+\bbeta_{X'} = f^*(K_X+B+\bbeta_X)$. Then there exists an effective $\mathbb{R}$-divisor $\Delta'\geq 0$ such that $\Delta'\equiv_{X} \bbeta_{X'}$ and $(X',B_{X'}+\Delta')$ is sub-klt, and $(X,B+f_*\Delta')$ is klt.
\end{lemma}
\begin{proof}
By taking a further resolution, we may assume that
$$\bbeta_{X'} = f^*(K_X+B+\bbeta_X) - (K_{X'}+B_{X'}) \equiv_{X} -(K_{X'}+B_{X'})$$
is nef and big over $X$, and that $f: X'\to X$ is a projective morphism. By Kodaira's lemma there exist an effective $\mathbb{R}$-divisor $E$ on $X'$
and, for every $n\gg 0$, an $f$-ample $\mathbb{R}$-divisor $H_{n}\geq 0$ on
$X'$ such that
$$
  -(K_{X'}+B_{X'}) \sim_{X,\mathbb{R}} H_{n}+\tfrac{1}{n}E .
$$
Fix $n$ sufficiently large and replace $H_{n}$ by a general effective
$\mathbb{R}$-divisor. Setting
$\Delta' := H_{n}+\tfrac{1}{n}E$, the pair $(X',B_{X'}+\Delta')$ is sub-klt,
and hence $(X,B+f_{*}\Delta')$ is klt.
\end{proof}

\begin{lemma}\label{l-kltbig} Let $X$ be a compact complex analytic variety in Fujiki's class $\mathcal C$ such that $(X,B+\bbeta)$ is generalized klt, and $B+\bbeta _X$ is modified big. Then  there is a K\"ahler resolution $\nu: X'\to X$ and a generalized klt pair $(X,G+\ggamma)$ with a K\"ahler form $\omega'$ on $X'$ such that $K_{X}+B+\bbeta _{X}\equiv K_{X}+G+\ggamma _{X}$, $\ggamma$ descends to $X'$, and $\ggamma _{X'}-\omega'$ is nef. \end{lemma}
\begin{proof} The proof is similar to  \cite[Lemma 3.1]{DH24}. Let $\nu :X'\to X$ be a K\"ahler log resolution of $(X,B+\bbeta _X)$. Since $B+\bbeta _X$ is modified big, then $B'+\bbeta _{X'}+E$ is big for some effective $\nu$-exceptional $\R$-divisor $E$ (see Lemma \ref{l-modnef}). 
By Demailly’s regularization theorem there is a
K\"ahler current $T$ with weakly analytic singularities such that $B'+\bbeta _{X'}+E \equiv  T$. By \cite[Lemma 3.7]{DH23} we may assume that $T\geq \omega '$
where $\omega '$ is a K\"ahler form and there is a resolution $\mu :X''\to X'$ such that $\mu ^*T=\mu ^* \omega '+\Theta +F$ where $F$ is an effective $\R$-divisor and $\Theta$ is a closed
nef (1,1)-current. Let $H$ be an effective $\mu$-exceptional $\R$-divisor such that $[\mu ^*\omega '-H]$ is a K\"ahler class and hence so is $\omega '':=\mu ^*\omega '-H+\Theta$. We write $K_{X''}+B''+\bbeta _{X''}=\mu ^*(K_{X'}+B'+\bbeta _{X'})$. Then 
\[B''+\bbeta _{X''}= \mu ^*(B'+\bbeta _{X'})-K_{X''/X'}\equiv -K_{X''/X'}-\mu ^*E+H+F+\omega ''.\] Therefore $B''+\bbeta _{X''}\equiv E''+\omega ''$ where $E''$ is an $\R$-divisor such that $\nu_*\mu_*E''=\nu_*\mu_*F\geq 0$.
We now let $\ggamma =(1-\delta)\bbeta+\delta \overline{\omega ''}$ and $G''=(1-\delta)B''+\delta E''$ for $0<\delta \ll 1$. We then have $G:=\nu _*\mu_* G''\geq 0$, $\lfloor G''\rfloor\leq 0$ and \[K_{X''}+G''+\ggamma _{X''}\equiv K_{X''}+B''+\bbeta _{X''}\equiv \mu ^*\nu ^*(K_{X}+B+\bbeta _X)\] so that $(X,G+\ggamma)$ is a generalized klt pair such that $K_X+B+\bbeta _X\equiv K_X+G+\ggamma _X$, and $\ggamma _{X''}- \delta \omega ''=(1-\delta)\bbeta _{X''}$ is nef. After replacing $X'$ by $X''$ and $\omega'$ by $\delta\omega''$, we have $\ggamma_{X'}- \omega'$ is nef.
\end{proof}
\begin{remark}
Note that we may replace $X'$ by any higher model $\rho:X''\to X'$ such that there is a $\rho$-exceptional divisor $F$ where $-F$ is $\rho$-ample. It suffices to define $\omega ''=\rho ^*\omega'-\epsilon F$, $\tilde \ggamma =\ggamma -\epsilon \bar F$ and $\tilde G''=G''+\epsilon F$ for $0<\epsilon \ll 1$. It then easily follows that $\omega ''$ is K\"ahler, $\tilde \ggamma _{X''}-\omega ''=\rho ^*(\ggamma _{X''}-\omega ')$ is nef,  $(X,\tilde G+\tilde \ggamma)$ is generalized klt where $\tilde G=\nu_*\rho _*\tilde G''$ and $K_X+B+\bbeta _X\equiv K_X+\tilde G+\tilde \ggamma$.
\end{remark}

\subsection{Mori cone and generalized Mori cone}

Let $f : X \to Y$ be a projective morphism of complex analytic varieties, let
$W \subset Y$ be a compact subset, and let $U$ be an open neighborhood of $W$.
Let $Z_1(X/Y;W)$ be the free abelian group generated by the projective integral
curves $C$ on $X$ with $f(C)$ is a point in $W$. We define
$$
\widetilde{A}(U, W) := \operatorname{Pic}\!\left(f^{-1}(U)\right) \big/ {\equiv_W},
$$
where two line bundles $L_1, L_2 \in \operatorname{Pic}(f^{-1}(U))$ are \textit{numerically equivalent over $W$} (written $L_1 \equiv_W L_2$) if $L_1 \cdot C = L_2 \cdot C$ for every curve $C \in Z_1(X/Y;W)$. We then define
$$
A^1(X/Y;W) := \varinjlim_{W \subset U} \widetilde{A}(U, W),
$$
where $U$ ranges over all open neighborhoods of $W$, and set ${\rm{NS}}(X/Y;W)_{\mathbb{R}} = N^1(X/Y;W) := A^1(X/Y;W) \otimes_{\mathbb{Z}} \mathbb{R}$. Let $N_1(X/Y;W)$ be the dual vector space of $N^1(X/Y;W)$. To distinguish the $N_1(X / Y ; W)$ introduced here from the one introduced below (see Remark \ref{l-anvsalg}), we sometimes add a superscript `alg' and write it as $N_1^{\text{alg}}(X / Y ; W)$. We can now define the relative Mori cone inside $N_{1}^{\mathrm{alg}}(X/Y;W)$.
\begin{definition}[Relative Mori cone]\label{def-Moricone}
With the notation and the setting as above, the \textit{relative Mori cone} $\overline{\mathrm{NE}}(X/Y;W)$ of $f : X \to Y$ over $W$ is the closure of the convex cone in $N_1(X/Y;W)$ spanned by the classes of projective integral curves $C$ such that $f(C)$ is a point in $W$. When $Y$ is compact, we can take $Y=W$ and write the relative Mori cone as $\overline{\text{NE}}(X/Y)$.
\end{definition}

The following version of the cone theorem will be useful for our purposes. 

\begin{theorem}\label{t-projcone} Let $f:X\to Y$ be a projective morphism of compact analytic varieties, $(X,B+\mathbf N)$ a generalized klt pair, where $\mathbf N$ is an $\R$-Cartier b-divisor.
\begin{enumerate}\item If $K_X+B+\mathbf N_X$ is not $f$-nef, then there exists a countable collection of $f$-vertical rational curves $C_i$ such that 
\[\overline{\rm NE}(X/Y)=\overline{\rm NE}(X/Y)_{K_X+B+\mathbf N_X\geq 0}+\sum _{i\in I}\R ^+[C_i]\] where $0<-(K_X+B+\mathbf N_X)\cdot C_i\leq 2\dim X$, for every extremal ray $\R ^+[C_i]$ there exists a line bundle $L_i$ such that $\R ^+[C_i]=\overline{\rm NE}(X/Y)\cap L_i^\perp$, and if $B+\mathbf N_X$ is $f$-big then $I$ is finite.
\item For every extremal ray $\R ^+[C_i]$ there is a contraction morphism $\varphi :X\to Z$ with a projective morphism $g:Z\to Y$ such that $-(K_X + B+\mathbf N_X)$ is $\varphi$-ample
and for any curve $C \subset  X$ such that $f (C)$ is a point we have that
\[\varphi (C)={\rm pt.}\qquad {\text{if and only if}}\qquad [C]\in \R^+[C_i].\]  
\end{enumerate}
Assuming $X$ is strongly $\Q$-factorial, we also have the following.
\begin{enumerate}
\item[(3)] We can run the projective $(K_X+B+\mathbf N_X)$-MMP over $Y$ with scaling of an $f$-ample divisor, which preserves the strongly $\Q$-factorial property.
\item[(4)] If $K_X+B+\mathbf N_X$ is pseudo-effective over $Y$ and either $B+\mathbf N_X$ or $K_X+B+\mathbf N_X$ is $f$-big, then any MMP with scaling of an $f$-ample line bundle terminates with a good relative minimal model.
\item[(5)] If $K_X+B+\mathbf N_X$ is not pseudo-effective over $Y$, then any MMP with scaling of an $f$-ample line bundle terminates with a Mori fiber space.
\end{enumerate}
%In particular, if $\bbeta \equiv _Y \mathbf N$ is a b-nef form,
\end{theorem}
%\Xieinline{The theorem itself doesn't need $X$ to be strongly $\Q$-factorial, shall we emphasize that the gmm and mfs in (4) and (5) are strongly $\Q$-factorial if $X$ is? If yes, I will then make the change. CH: I don't like calling it an mmp if it is not $\Q$-factorial, so maybe let's assume strongly $\Q$-factorial in (3-5). Of course it immediately follows that min models+MFS exist in all cases.}
\begin{proof}
We can write $Y=\bigcup_{i=1}^r U_i$ as a union of finitely many relatively compact Stein open subsets and $X_i:=f^{-1}\left(U_i\right)$. Let $H$ be any $f$-ample $\R$-line bundle on $X$. 

Now for any fixed $i$, since $U_i$ is relatively compact and Stein, by \cite[Theorem 7.2]{Fuj22} (a standard perturbation argument gives us the g-pair version), we have
\[
\overline{\mathrm{NE}}\left(X_i / U_i\right)=\overline{\mathrm{NE}}\left(X_i / U_i\right)_{K_X+B+\mathbf N_X+H\geq 0} + \sum_j \mathbb{R}_+ [\Gamma_{ij}],
\] where the last sum is finite, and $\Gamma_{ij}$ are rational curves with $0<-(K_X+B+\mathbf N_X+H)\cdot \Gamma_{ij}\leq 2\dim X$. By letting $H\to 0$ and taking the limit we get
$$
\overline{\mathrm{NE}}\left(X_i / U_i\right)
=
\overline{\mathrm{NE}}\left(X_i / U_i\right)_{K_X+B+\mathbf N_X\geq 0}
+
\sum_j \mathbb{R}_+ [\Gamma_{ij}]
$$
and the last sum becomes countable. Notice that every $f$-vertical curve is contained in some $X_i$. Therefore the canonical map
\[
\oplus_{i}\psi_i: \bigoplus_i\overline{\mathrm{NE}}(X_i / U_i) \to \overline{\rm NE}(X/Y)
\]
is surjective, one should also note that each $\psi_i: \overline{\mathrm{NE}}(X_i / U_i) \to \overline{\rm NE}(X/Y)$ is not necessarily injective. Then (1) follows immediately and (2) is a consequence of the usual relative base point free theorem for generalized klt pairs.

For (3), note that the contraction theorem holds by (2). If $f:X\to Z$ is a flipping contraction, then the existence of the flip can be checked locally over $Z$, and hence we may assume that $Z$ is relatively compact and Stein, $X\to Z$ satisfies property P, and $K_X+B+\mathbf N_X\equiv K_X+\Delta$ where $(X,\Delta)$ is klt. By \cite{Fuj22},  the log canonical model $X^+\to Z$ for $K_X+\Delta$ over $Z$ exists. This is also a log canonical model for $K_X+B+\mathbf N_X$ over $Z$ as required (see e.g. Lemma \ref{l-lcms}). Thus we can run the projective $(K_X+B+\mathbf N_X)$-MMP with scaling of an $f$-ample divisor by some standard arguments. Since the contractions are projective and extremal, by using \cite[Proposition 12.1.4]{KM92} as in the proof of \cite[Proposition 4.2]{HX26}, we can show the strongly $\Q$-factorial property is preserved in the MMP. 

For (4) and (5), we need to show the termination of the MMP in these situations, which can also be checked locally via \cite{Fuj22}.
\end{proof}

% We say that $\pi : X \to Y$ and $W$ satisfy \textit{Condition~(P)} if
% \begin{enumerate}[label=(P\arabic*)]
%     \item $X$ is a normal complex variety,
%     \item $Y$ is a Stein space,
%     \item $W$ is a Stein compact subset of $Y$,
%     \item $W \cap Z$ has only finitely many connected components for any analytic subset $Z$ defined on an open neighborhood of $W$.
% \end{enumerate}

%Cone theorem stated below shows that the restriction of the Mori cone to the negative half-space of the log canonical divisor is generated by a discrete set of negative extremal rays.

% \begin{theorem}[{\cite[Theorem~7.2]{Fuj22}}]\label{t-pcone}
% Let $(X, B+\bbeta)$ be a klt pair. Let $\pi : X \to Y$ be a projective morphism of complex analytic spaces, and let $W$ be a compact subset of $Y$ such that $\pi : X \to Y$ and $W$ satisfy~\textnormal{(P)}. Then
% $$
% \overline{\mathrm{NE}}(X/Y;W) = \overline{\mathrm{NE}}(X/Y;W)_{K_X+B +\bbeta \geq 0} + \sum_j \mathbb{R}^+[\Gamma_j].
% $$
% Moreover, if $\omega$ is relative K\"ahler over $Y$, then there are only finitely many extremal rays $\mathbb{R}^+[\Gamma_j]$ contained in $(K_X+B+ \bbeta+\omega)_{<0}$.
% \end{theorem}
% \begin{proof}
% This is clear when $\pi$ is bimeromorphic, since in that case we may replace 
% $\bbeta$ by a nef and big divisor and reduce to an ordinary klt pair.
% This suffices for our purposes (says proof of the main result). 

% \end{proof}
We next shift our attention to the \emph{generalized Mori cone} for compact complex analytic varieties. Let $X$ be a normal compact complex analytic variety. We set $N^1(X):= H^{1,1}_{\mathrm{BC}}(X)$. 
In contrast to the projective case, we define  $N_1^{\mathrm{an}}(X)$ as the space of real closed currents of bidimension $(1,1)$ modulo the equivalence relation $T_1 \equiv T_2$ if and only if
$$
T_1(\eta) = T_2(\eta)
$$
for all real closed $(1,1)$-forms $\eta$ admitting local potentials (when there is no ambiguity, we sometimes drop the superscript `an' in $N_1^{\text{an}}(X)$).

\begin{remark}\label{l-anvsalg}
The analytically defined $N_1^{\rm an}(X)$ is finer than $N_1^{\rm alg}(X)$: two curves may have $[C]=[C']$ in $N_1^{\rm alg}(X)$, while $[C]\neq[C']$ in $N_1^{\rm an}(X)$.

Indeed, consider the following example. Let $S_1,S_2$ be K3 surfaces of Picard number one admitting elliptic fibrations $f_i:S_i\to \mathbb P^1$, and let $F_i$ be a smooth fiber of $f_i$. Such surfaces exist: by the surjectivity of the period map \cite[Chapter 7]{Huy16}, there is a K3 surface $S$ with $\operatorname{NS}(S)=\mathbb Z e$ and $e^2=0$. By Riemann--Roch, after replacing $e$ by $-e$, we have $h^0(S,\mathcal O_S(e))\geq 2$.  After Stein factorization, the moving part $M$ of the linear system $|e|$ defines a fibration $S\to\mathbb P^1$, whose smooth fibers are elliptic curves by adjunction. Since $F_i^2=0$ and $[F_i]\neq 0$ in $\operatorname{NS}(S_i)$, we have
$$
\operatorname{NS}(S_i)_{\mathbb Q}=\mathbb Q[F_i],
$$
and in particular $M\cdot F_i=0$ for every line bundle $M$ on $S_i$. Hence $S_i$ is non-projective  K3 surface.

Set $X:=S_1\times S_2$ with natural projections $\mathrm{pr}_i: X \to S_i$, and choose points $p_i\in S_i$. Define
$$
C:=F_1\times p_2,\qquad C':=p_1\times F_2.
$$

We claim that $[C]=[C']=0$ in $N_1^{\rm alg}(X)$. Let $L$ be a line bundle on $X$, and let $L_1:=L|_{S_1\times\{p_2\}}$, viewed as a line bundle on $S_1$. Then
$$
L\cdot C=L_1\cdot F_1=0,
$$
and similarly $L\cdot C'=0$ by restricting $L$ to $\{p_1\}\times S_2$.

On the other hand, $C$ and $C'$ are distinct in $N_1^{\rm{an}}(X)$. Let $\omega_1$ be a K\"ahler form on $S_1$, and set $\eta:=\operatorname{pr}_1^*\omega_1$, so that $[\eta]\in H^{1,1}_{\rm BC}(X,\mathbb R)$. Then
$$
\int_C\eta=\int_{F_1}\omega_1>0,\qquad \int_{C'}\eta=0.
$$
Thus $[C]\neq [C']$ in $N_1^{\rm an}(X)$, although $[C]=[C']$ in $N_1^{\rm alg}(X)$.
\end{remark}

We next introduce the generalized Mori cone and the transcendental Mori cone. 
\begin{definition}[Generalized Mori cone and transcendental Mori cone]
The \textit{generalized Mori cone} $\overline{\mathrm{NA}}(X) \subset N_1(X)$ 
is the closed cone generated by the classes of closed positive $(1,1)$-currents. We define the \textit{transcendental Mori cone} $\overline{\mathrm{NE}}_{\mathrm{an}}(X)$ to be the closed subcone of $\overline{\mathrm{NA}}(X)$ generated by the integration currents associated with irreducible curves. More precisely, for an irreducible curve $C \subset X$, we denote by
$$
[C]: H^{1,1}_{\mathrm{BC}}(X)\ni \eta \longmapsto \int_C \eta
$$
the corresponding integration current, and set
$$
\overline{\mathrm{NE}}_{\mathrm{an}}(X)
:=\overline{\sum_i \mathbb{R}_{\geq 0}[C_i]}
\subset \overline{\mathrm{NA}}(X),
$$
where the sum runs over all irreducible curves $C_i \subset X$. There is a natural projection
$$
N_1^{\mathrm{an}}(X) \longrightarrow N_1^{\mathrm{alg}}(X),
$$
obtained by restricting functionals from $H^{1,1}_{\mathrm{BC}}(X)$ to $\mathrm{NS}(X)_{\mathbb{R}}$. Under this projection, the transcendental Mori cone $\overline{\text{NE}}_{\rm{an}}(X)$ maps onto the usual Mori cone $\overline{\text{NE}}(X)$.
\end{definition}

The generalized Mori cone also satisfies the cone theorem: the restriction of the generalized Mori cone to the negative half-space of the adjoint class is generated by a discrete set of negative extremal rays.
\begin{theorem}
Let $(X,B+\bbeta )$ be a  compact K\"ahler generalized klt pair. Then there are at most countably many rational curves $\{\Gamma _i\}_{i\in I}$ such that $0< -(K_X+B+\bbeta _X)\cdot \Gamma _i\leq 2\dim X$ for all $i\in I$ and  \[\overline{\rm NA}(X)=\overline{\rm NA}(X)_{(K_X+B+\bbeta _X) \geq 0}+\sum _{i\in I}\mathbb R ^+[\Gamma _i].\]
Moreover, if $B+\bbeta _X$ or $K_X+B+\bbeta _X$ is big, then $I$ is finite.
\end{theorem}
The strongly $\Q$-factorial case is proven in \cite{HP24}, the proof of the general case above will be given in the next section (Corollary \ref{c-conetheorem}), after developing the necessary tools.

\subsection{Divisorial contractions, flipping contractions}
In the minimal model program, each negative extremal ray $\mathbb{R}_{\geq 0}[\Gamma_i]$ appearing in the cone theorem is associated with an analytic contraction $f_{\Gamma_i}: X\to Z$ of divisorial, flipping, or fiber type.

\begin{definition}[Divisorial contractions]
Let $(X,B+\boldsymbol{\beta})$ be a generalized pair, where $X$ is a strongly $\Q$-factorial.
A $(K_X+B+\boldsymbol{\beta}_X)$-\emph{divisorial contraction} $f:X\to Z$ is a  bimeromorphic morphism of compact varieties, with relative Bott--Chern Picard number $\rho_{\rm{BC}}(X/Z)= 1$ and exceptional locus being an (irreducible) divisor, such that $-(K_X+B+ \boldsymbol{\beta}_X)$ is K\"ahler over $Z$. If $\Gamma$ is an extremal ray of $\overline{\text{NA}}(X)$ such that a curve $C$ is contracted by $f$ if and only if $[C]\in \Gamma$, then we say that $f:X\to Z$ is the divisorial contraction associated to $\Gamma$.
\end{definition}

% \Xieinline{ 

% (iii) CH: I don't think we need compact. LX: But it seems that we didn't define $\overline{\rm NA}(X)$ etc. well in the rel compact case... YL: You are right. looks not so obvious.
% CH: I add compact! Remove this comments.}

\begin{definition}[Flipping contractions, flips]
Let $(X,B+\boldsymbol{\beta})$ be a generalized pair, where $X$ is strongly $\Q$-factorial.
A $(K_X+B+\boldsymbol{\beta}_X)$-\emph{flipping contraction} $f: X \rightarrow Z$ is a small bimeromorphic morphism (i.e., an isomorphism in codimension 1) of compact varieties such that $\rho_{\rm{BC}}(X / Z)=1$ and $-\left(K_X+B+ { \boldsymbol{\beta}}_X\right)$ is K\"ahler over $Z$. If $\Gamma$ is an extremal ray of $\overline{\text{NA}}(X)$ such that a curve $C$ is contracted by $f$ if and only if $[C]\in \Gamma$, then we say that $f:X\to Z$ is the flipping contraction associated to $\Gamma$. The corresponding \emph{flip} (if it exists) is a small bimeromorphic morphism $f^{+}: X^{+} \rightarrow Z$ such that $\rho_{\rm{BC}}\left(X^{+} / Z\right)=1$ and $K_{X^{+}}+B^{+}+ \boldsymbol{\beta}_{X^{+}}$ is  K\"ahler over $Z$, where $B^{+}$ is the strict transform of $B$.
\end{definition}

One can prove that there exists a nef supporting class associated to the divisorial/flipping contraction.

\begin{proposition}\label{p-nefsupp}
Let $(X,B+ \bbeta)$ be a compact K\"ahler generalized  klt pair, and let $f:X\to Z$ be a divisorial/flipping contraction of an extremal ray $\Gamma$. Then there exists a nef class $\alpha \in H^{1,1}_{\rm{BC}}(X)$ such that $\alpha-(K_X+B+\bbeta _X)$ is K\"ahler and
$$
\Gamma=\{z \in \overline{\rm NA}(X) \mid \alpha \cdot z=0\}.
$$
\end{proposition}
\begin{proof}
Since $\Gamma$ is an extremal ray, there is a class $\alpha$ such that $\Gamma =\overline{\rm NA}(X)\cap \alpha ^\perp$ and by possibly replacing $\alpha$ by $-\alpha$, we may assume $\alpha\cdot \gamma\geq 0$ for any $\gamma\in \overline{\rm NA}(X)$. By the proof of \cite[Claim 3.23]{DHY23}, one easily sees that $\alpha -\delta (K_X+B+\bbeta _X)$ is positive on $\overline{\rm NA}(X)$ for any $0<\delta \ll 1$. But then $\alpha -\delta (K_X+B+\bbeta _X)$ is K\"ahler. 
Let $\alpha '=\frac 1 \delta \alpha$, then $\alpha '$ is nef, $\Gamma =\overline{\rm NA}(X)\cap (\alpha') ^\perp$, and 
$\alpha'-(K_X+B+\bbeta _X)=\frac 1\delta(\alpha -\delta(K_X+B+\bbeta _X))$ is K\"ahler.
Replacing $\alpha$ by $\alpha'$, the claim follows.
\end{proof}

Divisorial contractions and flips preserve dlt singularities.
\begin{lemma}\label{l-dltmmpdlt}
If $X$ is strongly $\Q$-factorial and $(X,B+\bbeta)$ is gdlt and $\phi:X\dasharrow X'$ is a flip or divisorial contraction, then $(X',B'+\bbeta )$ is strongly $\Q$-factorial and gdlt. 
\end{lemma}

\begin{proof}
Let $U\subset X$ be the biggest open subset such that $\phi |_U$ is an isomorphism and $U'=\phi (U)$.
Since $(X,B+\bbeta)$ is gdlt, there is an open subset $V\subset X$ such that $(V,(B+\bbeta )|_V)$ is a log resolution of itself, $-1\leq a(P,X,B+\bbeta)\leq 0$ for any prime divisor $P$ on $V$, and $-1<a(P,X,B+\bbeta)$ for any prime divisor $P$ over $X$ with center contained in $Z:=X\setminus V$. Let $V'=\phi(V\cap U)$, then
$(V', (B'+\bbeta ')|_{V'})$  is a log resolution of itself. The only thing to prove is that $a\left(P, X^{\prime}, B^{\prime}+\bbeta^{\prime}\right)>-1$ for any prime divisor $P$ over $X'$ whose center is contained in $Z':= X' \setminus V'$. For those $P$, the center $c_X(P)$ on $X$ cannot lie inside $V\cap U$. Thus, there are two possible cases: $c_X(P) \subset X \setminus  V$ or $c_X(P)\subset \operatorname{Ex}(\phi)$. In the first case, we have $a\left(P, X^{\prime}, B^{\prime}+\bbeta^{\prime}\right) \geq a(P, X, B+\bbeta)>-1$, while in the second case, we have $a\left(P, X^{\prime}, B^{\prime}+\bbeta^{\prime}\right)>a(P, X, B+\bbeta) \geq-1$. Thus $a(P,X',B'+\bbeta')>-1$ for all $P$ with center contained in $Z'$ and so $(X',B'+\bbeta')$ is gdlt. The strong $\Q$-factoriality of $X'$ follows from \cite[Lemma 2.5]{DH25} (see also the proof of \cite[Proposition 4.2]{HX26}).
 \end{proof}
%We will also need the following relative version of the cone theorem. \begin{theorem}[{\cite[Theorem 0.5]{HP24}}]\label{t-conerel} Let $f:X\to Y$ be a contraction morphism of compact normal K\"ahler varieties such that $\dim X=n$, $X$ is  $\mathbb{Q}$-factorial, and $(X, B+\boldsymbol{\beta})$ is generalized klt, then there are at most countably many $f$-vertical rational curves $\left\{\Gamma_i\right\}_{i \in I}$ such that$$\overline{\mathrm{NA}}(X/Y)=\overline{\mathrm{NA}}(X/Y)_{K_X+B+\boldsymbol{\beta}_X \geq 0}+\sum_{i \in I} \mathbb{R}^{+}\left[\Gamma_i\right]$$where $0<-\left(K_X+B+\boldsymbol{\beta}_X\right) \cdot \Gamma_i \leq 2 n$. \end{theorem}\begin{proof}\end{proof}
\subsection{Log minimal models and log canonical models}

In this subsection, we introduce log minimal models and log canonical models.
\begin{definition}[Log minimal models, good minimal models]
Let $(X,B+\bbeta)$ be a generalized klt pair over $S$. We say that a proper bimeromorphic map $\phi: (X,B+ \bbeta)\dasharrow (X',B'+ \bbeta)$ over $S$ is a \textit{log minimal model} if 

(a) $(X',B'+ \bbeta)$ is strongly $\Q$-factorial generalized klt pair,

(b) $\phi$ is a bimeromorphic contraction (and hence extracts no divisors),

(c) $K_{X'}+B'+ \bbeta_{X'}$ is nef over $S$,

(d) for each $\phi$-exceptional divisor $E$, the discrepancy satisfies $a(E,X,B+\bbeta)< a(E,X',B'+ \bbeta')$.

If furthermore $K_{X'}+B'+\bbeta'$ is semi-ample over $S$, then $\phi:  (X,B+ \bbeta)\dasharrow (X',B'+ \bbeta')/S$ is a good minimal model. 
\end{definition}

\begin{definition}[{Weak log canonical models, log canonical models, \cite[Definition 2.11]{DHY23}}]\label{def-lcm}
If $(X , B+\boldsymbol{\beta})$ is a generalized dlt pair over $S$, then we say that a bimeromorphic map $\phi: X \dasharrow X^m$ (proper over $S$) is a weak log canonical model over $S$ (resp. a log canonical model over $S$) if (1-3) below hold (resp. (1-4) below hold).

(1) $(X^m, B^m+\boldsymbol{\beta})$ is a generalized lc pair, where $B^m=\phi_* B+E$, and $E$ is the reduced sum of all $\phi^{-1}$-exceptional divisors,

(2) $K_{X^m}+B^m+\boldsymbol{\beta}_{X^m}$ is nef over $S$,

(3) $a(P, X, B+ \boldsymbol{\beta}) \leq a\left(P, X^m, B^m+ \boldsymbol{\beta}\right)$ for every $\phi$-exceptional divisor $P$, and

(4) $\left[K_{X^{ m}}+B^m+\boldsymbol{\beta}_{X^m}\right] \in H_{\mathrm{BC}}^{1,1}\left(X^m\right)$ is a K\"ahler class.
\end{definition}

\begin{lemma}\label{l-lcms}
Let $(X, B + \bbeta)$ be a generalized klt pair and $(X, B + \Delta)$ be a klt pair. Assume that $K_X + B + \bbeta_X \equiv K_X + B + \Delta$. Then $\phi : X \dasharrow Z$ is a generalized log canonical model for $K_X + B + \bbeta_X$ if and only if $\phi:X\dasharrow Z$ is a log canonical model for $K_X + B + \Delta$.
\end{lemma}
\begin{proof}
Let $B_Z = \phi_* B$, $\Delta_Z = \phi_*\Delta$, $\bbeta_Z = \phi_*\bbeta_X$, and let $p: W \to X$ and $q: W \to Z$ be a common resolution of $\phi: X \dasharrow Z$. We write
$$p^*(K_{X}+ B + \Delta) = q^*(K_Z+B_Z+\Delta_Z)+E, \ \ p^*(K_{X}+ B + \bbeta_X) = q^*(K_Z+B_Z+\bbeta_Z)+F.$$
Since $K_X+B+\Delta \equiv K_X+B+\bbeta_X$, the negativity lemma \cite[Lemma 2.8]{DH24} gives $E = F$, hence
$$K_Z+B_Z+\Delta_Z \equiv K_Z+B_Z+\bbeta_Z.$$
In particular, $K_Z+B_Z+\Delta_Z$ is ample if and only if $K_Z+B_Z+\bbeta_Z$ is K\"ahler, and $\phi$ is $(K_X+B+\Delta)$-non-positive if and only if it is $(K_X+B+\bbeta_X)$-non-positive. The result then follows from Definition~\ref{def-lcm}.
\end{proof}

We next prove an existence result for minimal models. Before doing so, we introduce condition {(P)}, which will be used repeatedly throughout the rest of the paper. Let $f: X \to Y$ be a projective morphism of complex analytic spaces, and let $W \subset Y$ be a compact subset. We say that $(f,W)$ satisfies \textit{property (P)} if the following conditions hold:
\begin{enumerate}
\item $X$ is a normal complex variety,
\item $Y$ is a Stein space,
\item $W$ is a Stein compact subset of $Y$,
\item $W \cap Z$ has only finitely many connected components for any analytic subset $Z$ defined on an open neighborhood of $W$.
\end{enumerate}
This condition will be useful when running the relative MMP for projective morphisms between complex analytic spaces.

\begin{proposition}\label{p-relmmp} Let $f:X\to Y$ be a contraction morphism of normal compact complex varieties. Assume that $(X,B+\bbeta )$ is a generalized klt pair such that $B+\bbeta_X$ is modified big over $Y$ and $\bbeta \equiv _Y \mathbf N$ where $\mathbf N$ is an $\R$-Cartier b-divisor. Then
\begin{enumerate}\item If $K_X+B+\bbeta_X$ is pseudo-effective over $Y$, then $(X,B+\bbeta)$ has a $\Q$-factorial good minimal model over $Y$.
\item  If $K_X+B+\bbeta_X$ is not pseudo-effective over $Y$, then $(X,B+\bbeta)$ is birational to a Fano fibration  over $Y$.
\end{enumerate}
\end{proposition}
\begin{proof}  %By \cite[Theorem 3.1]{CH24}, $f$ is Moishezon. Since we can replace $X$ by a higher model, we may assume that $f$ is projective and $X$ is smooth.
Since $B+\bbeta _{X}$ is modified big over $Y$, there is a resolution $\nu :X'\to X$ such that if $K_{X'}+B'+\bbeta _{X'}=\nu^* (K_X+B+\bbeta _X)$, then $B'^{\geq 0}+\bbeta _{X'}+\epsilon E$ is big over $Y$ for  $E={\rm Ex}(\nu)$ and any $\epsilon >0$ (see Lemma \ref{l-modnef}).
Since $\mathbf N _{X'}\equiv_Y \bbeta _{X'}$, then $B'^{\geq 0}+\mathbf N _{X'}+\epsilon E$ is big over $Y$, and so we may assume that $f$ is Moishezon.
Replacing $(X,B+\bbeta )$ by a higher model $(X', B'^{\geq 0}+\epsilon E+\bbeta )$ for $0<\epsilon \ll 1$ (see \cite[Lemma 2.12]{DHY23}), we may assume that $f$ is projective and $\bbeta$ and $\mathbf N$ descend to $X$ where $B+\mathbf{N}_X$ is big over $Y$. It suffices to show that the generalized pair $(X,B+\mathbf{N})$ has a strongly $\Q$-factorial good minimal model or Mori fiber space over $Y$, which follows immediately from Theorem \ref{t-projcone}.

\end{proof}
\begin{lemma}\label{l-term}
Let $f:X\to U$ be a projective morphism between complex analytic spaces satisfying property (P), assume that $(X,B)$ is a klt pair, and $A$ a relatively ample line bundle. Assume that $(X,B)$ admits a relative good minimal model over $U$. There exists an $\epsilon>0$ such that if $g_t: X \dasharrow Z_t/U$ is the log canonical model of $(X, B+t A)$ over $U$, then $Z_t$ is independent of $t \in(0, \epsilon)$, and there is a morphism $Z_t \to Z_0/U$. 
\end{lemma}

\begin{proof}
The proof follows \cite[Lemma 2.9.2]{HMX18}. If $K_X+B$ is nef over $U$, then the claim is clear. Otherwise, we run the relative $(K_X+B)$-MMP with scaling of $A$ over $U$. Since $(X,B)$ admits a good minimal model, this MMP terminates with a good minimal model over $U$ $$X\dasharrow X^1 \dasharrow \cdots \dasharrow X^{N-1} \dasharrow   X'.$$

Note that each step of this MMP is a step of the $(K_X+B+tA)$-MMP over $U$ and $K_{X'}+B'+tA'$ is semi-ample over $U$ for any $t\in [0,\epsilon)$.

Since $K_{X'}+B'+tA'$ is semi-ample over $U$ for $t\in [0,\epsilon)$ we may consider the induced morphisms $X'\to Z_t/U$. Let $C$ be a relative curve for $X'\to U$.  If for some $0<t_0<\epsilon$ we have $(K_{X'}+B'+t_0A')\cdot C=0$, then by the nefness of $K_{X'}+B'+tA'$ over $U$ for $t\in [0,\epsilon)$, it is immediate to see that $(K_{X'}+B'+tA')\cdot C=0$ for all $t\in[0,\epsilon)$. Therefore, the $(K_X+B+tA)$-log canonical model $Z_t$ over $U$ is independent of $t\in(0,\epsilon)$. By the rigidity lemma \cite[Lemma 4.1.13]{BS95}, there exists a morphism from $Z_t\to Z_0/U$. 
\end{proof}

\begin{proposition} \label{c-relmmp} Let $f:X\to Y$ be a bimeromorphic morphism of normal compact complex varieties with rational singularities. Assume that $(X,B+\bbeta )$ is a generalized klt pair. Then
\begin{enumerate} \item  $(X,B+\bbeta)$ has a log canonical model $\psi :X\dasharrow Z$ over $Y$. If $Y$ is K\"ahler, then so is $Z$. If $K_X+B+\bbeta _X$ is nef over $Y$, then $\psi$ is a contraction morphism. %\footnote{YL: I also suggest mentioning that if $K_X+B+\bbeta_X$ is additionally nef, then $\psi$ is a contraction morphism, since this is used several times.}
\item If $X$ is strongly $\Q$-factorial and $X,Y$ are K\"ahler, then we can run a $(K_X+B+\bbeta_X)$-relative minimal model program over $Y$ which ends with a good minimal model over $Y$. \end{enumerate}
\end{proposition} \begin{proof} 
(1) Since $(X,B+\bbeta )$ is a generalized klt, then $X$ has rational singularities. Let $\nu: X'\to X$ be a projective resolution such that $\bbeta$ descends to $X'$, and write $K_{X'}+B'+\bbeta _{X'}=\nu ^*(K_X+B+\bbeta _X)$.
Let $Y=\cup Y_i$ be an open cover by relatively compact Stein open subsets such that $f_i:X_i\to Y_i$ satisfies property (P), where $X_i=f^{-1}(Y_i)$. Let $X'_i=\nu ^{-1}(X_i)$, $\nu _i=\nu |_{X_i'}$, $B'_i=B'|_{X_i'}$, $B_i=B|_{X_i}$, $\bbeta _{X'_i}=\bbeta _{X'}|_{X'_i}$, and $\bbeta _{X_i}=\bbeta _{X}|_{X_i}$.
By Lemma \ref{l-rel} we may choose $L'_i$ an $\R$-line bundle on $X'_i$ such that $\bbeta _{ X'_i}\equiv L'_i$. Since $L'_i$ is nef and big over $Y_i$, we may assume that $L'_i\equiv \Delta '_i$ where $( X'_i,B'_i+ \Delta' _i)$ is sub-klt. If $\Delta _i=(\nu  _i)_*\Delta' _i$, then $(X_i,B_i+\Delta _i)$ is klt. Let $\psi _i:X_i\dasharrow Z_i$ be the log canonical model for $(X_i, B_i+\Delta _i)$ over $Y_i$ which exists by \cite{Fuj22}. 
If $B_{Z_i}+\Delta _{Z_i}=(\psi _i)_*(B_{i}+\Delta _{i})$, then $(Z_i,B_{Z_i}+\Delta _{Z_i})$ is klt and hence $Z_i$ has rational singularities.
Since $K_{X_i}+B_i+\Delta _i\equiv _{Y_i} K_{X_i}+B_i+\bbeta _{X_i}$, by uniqueness of log canonical models, these glue together so that we obtain $\psi :X\dasharrow Z$ a log canonical model for $K_{X}+B+\bbeta _{X}$ over $Y$. Note that by \cite[Lemma 2.6.(1)]{DH24}, it follows that  
$K_Z+B_Z+\bbeta _Z:=\psi _*(K_X+B+\bbeta _X )$ represents a class in $H^{1,1}_{\rm BC}(Z)$.

If $Y$ is K\"ahler, then since $K_Z+B_Z+\bbeta _Z$ is K\"ahler over $Y$, it follows that $Z$  is K\"ahler.

If $K_X+B+\bbeta _X$ is nef over $Y$, then consider $p:W\to X$ and $q:W\to Z$ a common projective resolution. We must show that if $C\subset W$ is a $p$-exceptional curve, then it is also a $q$-exceptional curve.
By the negativity lemma, it follows that $p^*(K_X+B+\bbeta _X)=q^*(K_Z+B_Z+\bbeta _Z)$ where $B_Z=\psi _* B$.
Since $p_*C=0$, then \[0=p^*(K_X+B+\bbeta _X)\cdot C=q^*(K_Z+B_Z+\bbeta _Z)\cdot q_*C.\] Since $q_*C$ is vertical over $Y$ and $K_Z+B_Z+\bbeta _Z$ is K\"ahler over $Y$, it follows that $q_*C=0$ and hence $\psi$ is a morphism.

(2) Fix $\omega _Y$ a K\"ahler form on $Y$ such that $\omega _Y\cdot C>2\dim X$ for any curve $C$ on $Y$. Let $R$ be a $(K_X+B+\bbeta_X+f^*\omega _Y )$-negative extremal ray, then by the cone theorem \cite[Theorem 0.5]{HP24}, we may assume that $R=\mathbb R^+[C]$ where $(K_X+B+\bbeta _X)\cdot C\geq -2\dim X$, and therefore $f^*\omega_Y \cdot C=0$ so $f_*C=0$. By the cone theorem, we may assume that the supporting hyperplane for $R$ is a nef class $\alpha=[K_X+B+\bbeta _X+\omega]$ for some K\"ahler class $\omega$. 
The corresponding contraction $\psi:X\to Z$ exists and is locally projective.
%by (1) and is projective by \cite[Lemma 2.40]{HX26}.

If $\dim {\rm Ex}(\psi)=\dim X-1$, then this is a $(K_X+B+\bbeta_X)$-divisorial contraction and we may replace $X$ by $Z$.

If  $\dim {\rm Ex}(\psi)<\dim X-1$, then this is a { $(K_X+B+\bbeta_X)$-}flipping contraction.
%Arguing as above, we obtain log-canonical models $X_i^+\to Z_i$ for $(X_i,B_i+\bbeta _{X_i})$ over $Z_i$.  By uniqueness of log canonical models, these glue together to give $\psi ^+:X^+\to Z$.
By (1) there is $\psi ^+:X^+\to Z$ a log canonical model for $(X,B+\bbeta _{X})$ over $Z$.
Note that $X^+$ is K\"ahler over $Z$ and $Z$ is K\"ahler so that $X^+$ is K\"ahler. We may replace $X$ by $X^+$. 

By repeating this process and Lemma \ref{l-dltmmpdlt}, we may run a { $(K_X+B+\bbeta_X)$-}minimal model program over $Y$ with scaling of a K\"ahler class $\eta$ on $X$ such that $K_X+B+\bbeta _X+\eta$ is K\"ahler over $Y$. Arguing as above, we let $Y=\cup Y_i$ be an open cover by relatively compact Stein open subsets such that $f_i:X_i\to Y_i$ satisfies property (P), where $X_i=f^{-1}(Y_i)$. 
Since 
$\eta|_{X_i}\equiv_{Y_i} A_i$ where $A_i$ is an ample $\R$-divisor over $Y_i$, this MMP induces a sequence of weak log canonical models for $K_{X_i}+B_i+\Delta _i+tA_i$ over $Y_i$ for some $t\in [0,1]$, this process terminates by \cite[Theorem E]{Fuj22} (see Lemma \ref{l-term}).
Thus, the above relative minimal model program with scaling will terminate with a minimal model $X\dasharrow X'$ for $K_{X}+B+\bbeta _{X}$ over $Y$. Since $K_{X'}+B'+\bbeta _{X'}$ is nef and $B'+\bbeta _{X'}$ is big  over $Y$, by (1), $K_{X'}+B'+\bbeta _{X'}$ is semi-ample over $Y$.
\end{proof}

\subsection{Rationally connected fibers} In this subsection, we introduce several consequences of \cite{HM07} that will be used in the proof of our main results.
\begin{lemma}\label{l-kltrcc} Let $(X,B+\bbeta)$ be a generalized klt pair and $f:X'\to X$ be a proper birational morphism, then the fibers of $f$ are rationally chain connected.\end{lemma}
\begin{proof} 
The question is local on $X$ and so we may assume that $X$ is a relatively compact Stein open set.
We write $K_{X'}+B'+\bbeta_{X'} = f^*(K_X+B+\bbeta_X)$.
By taking a further resolution, we then have that $-(K_{X'} + B') \equiv \bbeta_{X'}$ is nef and big, and so we may pick $\Delta' \equiv -(K_{X'}+B')$ such that $(X',B'+\Delta')$ is sub-klt and hence $(X,B+\Delta)$ is klt where $\Delta=f_*\Delta '$. The claim now follows from \cite[Corollary 1.5]{HM07}.
\end{proof}
\begin{lemma}\label{l-pltrcc} Let $(X,B+\bbeta)$ be a generalized log canonical pair, and let
$\nu:X'\to X$ be a log resolution. Write
$K_{X'}+B'+\bbeta_{X'}=\nu^*(K_X+B+\bbeta_X)$, and suppose that
$\bbeta_{X'}$ is K\"ahler over $X$ and that $(B')^{=1}=S'$ is an
irreducible divisor. Then the general fibers of $\nu|_{S'}$ are
rationally connected.
\end{lemma}
\begin{proof}  Since the question is local over $X$, we may assume that $X$ is Stein and $\nu$ satisfies property (P). We may also assume that $S'$ is exceptional over $X$.
%Arguing as in \cite[Corollary 1.4.3]{BCHM10}, there exists a birational morphism $\nu :X'\to X$ such that $S$ is the only $\nu$-exceptional divisor and 

Since $-(K_{X'}+B')\equiv _X \bbeta _{X'}$ is K\"ahler, so is $-(K_{X'}+B'-\epsilon S')$ for any $0<\epsilon \ll 1$. We may pick effective $\R$-divisors $D'\equiv _X-(K_{X'}+B')$ and $\Delta '\equiv_X -(K_{X'}+B'-\epsilon S')$ such that $(X',B'+D')$ is sub-plt and $(X',B'-\epsilon S'+\Delta ')$ is sub-klt, $K_{X'}+B'+D'\equiv_X K_{X'}+B'+\bbeta _{X'}$ and $K_{X'}+B'-\epsilon S'+\Delta '\equiv_X K_{X'}+B'+\bbeta _{X'}$. 
Let $\Gamma '=(B')^{\geq 0}-\epsilon S'+\Delta '+\delta ({\rm Ex}(\nu)-S')$ where $0<\delta \ll \epsilon$, then $\Gamma'$ is an effective $\R$-divisor such that $(X',\Gamma ')$ is klt, and $K_{X'}+\Gamma'\equiv _XF$ where $F\geq 0$ and its support is the set of $\nu$-exceptional divisors distinct from $S'$. 
By \cite{Fuj22} we may run the $(K_{X'}+\Gamma ')$-MMP over $X$. Let $\phi :X'\dasharrow \bar X$ be the output of this MMP and $\mu :\bar X\to X$ be the induced map, then $\bar F:=\phi _*F\equiv _X K_{\bar X}+\bar \Gamma :=\phi _*(K_{X'}+\Gamma ')$ is nef over $X$.
By the negativity lemma, $\bar{F}=0$. Since this MMP is  $(K_{X'}+B'+D')$-trivial, it follows that $(\bar X,\bar B+\bar D)$ is plt and hence $K_{\bar S}+\bar B_S+\bar D_S:=(K_{\bar X}+\bar B+\bar D)|_{\bar S}$ is klt. Since $D'$ is ample over $X$, then $\bar D_S$ is big and hence $-K_{\bar S}\equiv _X\bar B_S+\bar{D}_{S}$ is big over $X$.  By \cite[Corollary 1.4]{HM07}, the  general fibers of $\mu |_{\bar S}$ are rationally connected.
\end{proof}

\section{Small $\mathbb{Q}$-factorializations and dlt modifications}

In this section, we first establish the existence of strongly $\Q$-factorial modifications in Theorem \ref{t-3}, and use these to construct dlt modifications for generalized pairs in Theorem \ref{t-dltmodel}. As an application, we prove in Corollary \ref{c-conetheorem} a version of the cone theorem from \cite{HP24}, which does not require the $\Q$-factorial assumption.

\begin{proposition}\label{t-2} Let $(X,B+\bbeta)$ be a generalized klt pair, where $X$ is compact and $L_1,\ldots,L_k$ are reflexive rank 1 sheaves. Then there exists a small projective  birational morphism $\nu :X'\to X$ such that $\nu ^{-1}_*L_1,\ldots ,\nu ^{-1}_*L_k$ are $\Q$-Cartier.
\end{proposition}
Note that as $\nu :X'\to X$ is small, ${\rm Ex}(\nu)$ has codimension $\geq 2$ and hence we may define $\nu ^{-1}_*L_i=\iota_*\bigl(L_i|_{U'}\bigr)$ where $U'=X'\setminus {\rm Ex}(\nu)$ is isomorphic to $U=\nu (U')$ and $\iota:U'\hookrightarrow X'$ is the inclusion, $L_i|_{U'}$ means the pull back of $L_i|_{U}$ along the isomorphism $\nu|_{U'}:U' \stackrel{\sim}\to U$.
\begin{proof} We follow \cite[Lemma 5.7]{DH23}. It suffices to prove the claim for $k=1$, i.e. we may assume that we have a single reflexive rank 1 sheaf $L$. Let  $X=\cup V_i$ where $V_i$ are relatively compact Stein spaces. We may write $L|_{V_i}=\OO _{V_i}(D_i)$ where $D_i$ is a Weil divisor. 
By \cite[Theorem 2.19]{DHY23} there is a projective small $\Q$-factorialization $\nu _i:V'_i\to V_i$ and an $\R$-divisor $\Delta _i'\equiv \bbeta _{V'_i}$ such that $K_{V'_i}+B_i'+\Delta '_i=\nu _i^*(K_{V_i}+B_i+\Delta _i)$ where $\Delta _i=(\nu _i)_*\Delta '_i$ and $(V_i,B_i+\Delta _i)$ is klt. 
Possibly shrinking $V_i$, we may assume that $\nu_i$ satisfies  property (P) and $K_{V_i}+B_i+\Delta _i\sim _\R0$ so that also
$K_{V'_i}+B'_i+\Delta' _i\sim _\R 0$. Let $D_i'=(\nu _i^{-1})_*D_i$, then $(V'_i,B'_i+\Delta ' _i+\epsilon D'_i)$ is klt for $0<\epsilon \ll 1$. 
By \cite[Theorem 1.8]{Fuj22}, $(V'_i,B'_i+\Delta ' _i+\epsilon D'_i)$ admits a log canonical model (over $V_i$). Since $K_{V'_i}+B'_i+\Delta ' _i+\epsilon D'_i\sim _\R \epsilon D'_i$, this log canonical model is independent of $\epsilon$.
Replacing $V'_i$ by this log canonical model we may assume that $K_{V'_i}+B'_i+\Delta ' _i+\epsilon D'_i$ is ample and klt for any $0<\epsilon \ll 1$ (we may no longer assume that $V'_i$ is $\Q$-factorial).
By uniqueness of log canonical models the $V'_i$ glue together to give $\nu:X'\to X$ a small birational morphism such that $L':=\nu ^{-1}_*L$ is a $\Q$-line bundle, ample over $X$. 
\end{proof}

 As a consequence of Proposition \ref{t-2}, we obtain the existence of strongly $\Q$-factorial modifications for arbitrary generalized klt pairs.
\begin{theorem}[Strong $\Q$-factorialization]\label{t-3} Let $(X,B+\bbeta)$ be a generalized klt pair, where $X$ is compact. Then there exists a small projective bimeromorphic morphism $\mu :X'\to X$ such that $X'$ is strongly $\Q$-factorial.
\end{theorem}
We denote by ${\rm Cl}_\R(X)$ the class group by which we mean the real vector space generated by divisorial sheaves modulo the real vector space generated by line bundles. If $X$ is quasi-projective or $X$ is Stein, then this agrees with the usual definition of the class group generated by $\R$-divisors modulo Cartier divisors. 
\begin{proof} We begin by observing that if $(X,B+\bbeta)$ is a generalized klt pair then for any point $x\in X$ there is an open subset $x\in V\subset X$ which is relatively compact and Stein, that admits a $\Q$-factorialization $\nu:V'\to V$ (cf. \cite[Theorem 1.24]{Fuj22}). In particular the divisor class group ${\rm Cl}_\R(V)\cong H^0(V,R^2\nu _*\R)$ is finite dimensional. To see this, note that since $V$ has rational singularities, $R^i\nu _*\OO _{V'}=0$ for $i>0$, and so $R^1\nu_*\OO _{V'}^*\to R^2\nu_*\R$ is an isomorphism. Let $M_1,\ldots , M_k$ be divisors on $V'$ corresponding to a basis of $H^0(V,R^2\nu _*\R)$.
If $D$ is a divisor on $V$, then $D':=\nu ^{-1}_*D\equiv \sum m_iM_i$ where $m_i\in \R$.
Since $K_{V'}+B'+\bbeta _{V'}\equiv K_{V'}+B'+\Delta'\equiv _V0$ where $(V',B'+\Delta ')$ is klt, then by the base point free theorem it follows easily that $D'\sim_\R\sum m_iM_i$, i.e. $D\sim_\R \sum m_i\nu _*M_i$. Therefore $\nu _* M_1,\ldots , \nu _* M_k$ generate ${\rm Cl}_{\R}(V)$.

%is  $\Q$-Cartier on $X'$ corresponding to an element in $H^1(\OO _{X'}^*)\otimes \Q \cong H^0(R^1\nu _* \OO _{X'}^*)\otimes \Q\cong H^0(R^2\nu _*\mathbb Q)$. Clearly $\nu_*$ In particular $H^0(R^2\nu _*\R)$ is generated by the classes of finitely many line bundles $M_1,\ldots , M_k$ on $X'$. Let $L$ be a reflexive rank 1 sheaf on $X$, then $(\nu ^{-1}_*L)^{[m]}$ is a line bundle for some integer $m>0$. Replacing $m$ by a multiple, we may assume that $(\nu ^{-1}_*L)^{[m]}\equiv M_1^{\otimes m_1}\otimes \ldots \otimes M_k^{\otimes m_k}$ where $m_i$ are integers. %But then $(\nu ^{-1}_*L)^{[-m]}\otimes  M_1^{\otimes m_1}\otimes \ldots \otimes M_k^{\otimes m_k}\equiv _X0$ and so By the base point free theorem, $(\nu ^{-1}_*L)^{[m]}\cong   M_1^{\otimes m_1}\otimes \ldots \otimes M_k^{\otimes m_k}$ and so . 

Let $X=\cup V_i$ where $V_i$ are relatively compact Stein spaces admitting a $\Q$-factorialization. By what we have seen above, each ${\rm Cl}_\R(V_i)$ is finite dimensional. Since this cover is finite, it follows that the image of $\Phi:{\rm Cl}_\R(X)\to \oplus {\rm Cl}_\R(V_i)$ is finite dimensional and defined over $\Q$. Let $L_1,\ldots ,L_r$ be reflexive rank 1 sheaves such that $\Phi(L_1),\ldots , \Phi(L_r)$ generate the image of $\Phi$. Let $\mu:X'\to X$ be the small birational morphism defined by Proposition \ref{t-2}, $V'_i=\mu ^{-1}(V_i)$, and $\mu _i=\mu |_{V'_i}$. In particular, $L'_i=\mu ^{-1}_*L_i$ are $\Q$-line bundles.
Let $L$ be a reflexive rank 1 sheaf on $X$, we show that $\mu ^{-1}_*L$ is $\R$-Cartier (and hence $\Q$-Cartier). We have that $L|_{V_i}\sim _\R \sum r_jL_j|_{V_i}$. It follows that $(\mu_i ^{-1})_*L\sim _\R \sum r_jL'_j|_{V'_i}$ is $\R$-Cartier.
\end{proof}
We can now construct dlt modifications for generalized pairs, following the strategy of \cite{Fil20} and \cite{HP24}.
\begin{theorem}[Global dlt-models]\label{t-dltmodel} 
Let $X$  be a compact analytic variety and  $(X,B+\bbeta )$ a generalized pair. Then there exists a birational projective morphism $f^{\rm m}:X^{\rm m}\to X$ such that $X^{\rm m}$ is strongly $\Q$-factorial, all $f^{\rm m}$-exceptional divisors $P$ have discrepancy $a(P,X,B+\bbeta )\leq -1$ and
 $(X^{\rm m}, B^{\rm m})$ is dlt where $B^{\rm m}={(f^{\rm m})}^{-1}_*(B\wedge {\rm Supp}(B))+{\rm Ex}(f^{\rm m})$.
\end{theorem}

%\begin{theorem}[LC models]\label{t-lcmodel} Let $(X,B+\bbeta )$ be a generalized pair, where $X$ is compact. Then there exists a projective birational morphism $f^{\rm m}:X^{\rm m}\to X$ such that all $f^{\rm m}$-exceptional divisors $P$ have discrepancy $a(X,B+\bbeta ,P)\leq -1$ and$(X^{\rm m}, B^{\rm m})$ is log canonical where $B^{\rm m}={(f^{\rm m})}^{-1}_*(B\wedge {\rm Supp}(B))+{\rm Ex}(f^{\rm m})$.\end{theorem}

\begin{proof}
The proof follows that of \cite[Theorem 3.2]{Fil20} (see also \cite{HP24}) where we use Theorem \ref{t-3} for the existence of $\Q$-factorializations of generalized klt pairs and Theorem \ref{t-projcone} to run the required relative MMP.
We include the details for the convenience of the reader.

Let $f:X'\to X$ be a log resolution of the generalized pair $(X,B+\bbeta)$ and write $K_{X'}+B'+\beta '=f^*(K_X+B+\beta)$ where $\beta =\bbeta _X$ and $\beta '=\bbeta _{X'}$. 
We may assume that $f$ is defined by a sequence of blow-ups along centers of codimension $\geq 2$  in $X$, and hence $f$ is a projective morphism, and so we have $C\geq 0$ an $f$-exceptional divisor such that $-C$ is relatively ample.
We write $B'=f^{-1}_*\{ B\}+E^++F-G$ where $E^+,F,G$ are supported on the divisors of discrepancy $a\leq -1,\ -1<a <0,\ a >0$ respectively and we let $E={\rm red}(E^+)$ be the reduced divisor with the same support as $E^+$.
For any $0<\epsilon, \mu, \nu<1$, we have
\[E+(1+\nu)F-\mu C+\beta '=(1-\epsilon \mu)E+(1+\nu)F+\mu(\epsilon E-C)+\beta '.\]
Note that $-\mu C+\beta '\equiv _X -\mu C-(K_{X'}+B')$ is an ample $\R$-divisor (over $X$) and hence for $0<\epsilon \ll 1$,
$\mu(\epsilon E-C)+\beta '$ is also numerically equivalent to an ample $\R$-divisor (over $X$). 

 We may write \[-\mu C+\beta '\equiv _X H_{1,\mu}\qquad  {\rm and}\qquad
\mu(\epsilon E-C)+\beta '\equiv _X H_{2,\mu}.\] Let $\Delta_{\epsilon, \mu,\nu}:=f^{-1}_*\{ B\}+(1-\epsilon \mu)E+(1+\nu)F$, then $(X',\Delta_{\epsilon, \mu,\nu}+H_{2,\mu})$ is generalized klt  for $0< \nu \ll 1$.
By Theorem \ref{t-3} and  Theorem \ref{t-projcone}, there is a $\Q$-factorial minimal model projective over $X$
\[\psi^{\rm m}_{\epsilon, \mu,\nu}:X'\dasharrow X^{\rm m}_{\epsilon, \mu,\nu},\qquad  f^{\rm m}_{\epsilon, \mu,\nu}:X^{\rm m}_{\epsilon, \mu,\nu}\to X.\]
Since \[\Delta_{\epsilon, \mu,\nu}+H_{2,\mu }\equiv _Xf^{-1}_*\{ B\}+E+(1+\nu)F+H_{1,\mu },\]
then $\psi^{\rm m}_{\epsilon, \mu,\nu}$ is also a $\Q$-factorial minimal model over $X$
for the generalized dlt pair $(X', f^{-1}_*\{ B\}+E+(1+\nu)F+H_{1,\mu })$. In particular, by Lemma \ref{l-dltmmpdlt}, $(X^{\rm m}_{\epsilon, \mu,\nu}, B^{\rm m}_{\epsilon, \mu,\nu}:= (\psi^{\rm m}_{\epsilon, \mu,\nu})_*(f^{-1}_*\{ B\}+E+F))$ is dlt.

Define \[N:= K_{X^{\rm m}_{\epsilon, \mu,\nu}}+B_{\epsilon, \mu,\nu}^{\rm m}+\nu F^{\rm m}_{\epsilon, \mu,\nu} + H^{\rm m}_{1,\epsilon, \mu,\nu}\equiv _X K_{X^{\rm m}_{\epsilon, \mu,\nu}}+\Delta_{\epsilon, \mu,\nu}^{\rm m},\]
\[T:=K_{X^{\rm m}_{\epsilon, \mu,\nu}}+B_{\epsilon, \mu,\nu}^{\rm m}+(E^+-E)^{\rm m}_{\epsilon, \mu,\nu}-G^{\rm m}_{\epsilon, \mu,\nu}+\beta ^{\rm m}_{\epsilon, \mu,\nu}\equiv_X 0 .\]
We then have
\[T-N\equiv _X \mu C^{\rm m}_{\epsilon, \mu,\nu}+(E^+-E)^{\rm m}_{\epsilon, \mu,\nu}-G^{\rm m}_{\epsilon, \mu,\nu}-\nu F^{\rm m}_{\epsilon, \mu,\nu}=:D^{\rm m}_{\epsilon, \mu,\nu},\]
where $-D^{\rm m}_{\epsilon, \mu,\nu}$ is nef over $X$ and the pushforward of $D^{\rm m}_{\epsilon, \mu,\nu}$ to $X$ is effective. By the negativity lemma, $D^{\rm m}_{\epsilon, \mu,\nu}\geq 0$.  
The divisors $C$, $E^+-E$, $F$ and $G$ are independent of ${\epsilon, \mu,\nu}$, thus if $0<\mu \ll \nu \ll 1$, then $G^{\rm m}_{\epsilon, \mu,\nu}=\nu F^{\rm m}_{\epsilon, \mu,\nu}=0$. We now suppress the subscript $(.)_{\epsilon, \mu,\nu}$ and let $X^{\rm m}:=X^{\rm m}_{\epsilon, \mu,\nu}$.
Then $(X^{\rm m},B^{\rm m})$ is dlt and $ B^{\rm m}={(f^{\rm m})}^{-1}_*(B\wedge {\rm Supp}(B))+{\rm Ex}(f^{\rm m})$.

\end{proof}

% \begin{proposition}\label{p-curve} Let $(X,B+\mathbf M )$ be a Moishezon generalized pair such where $\mathbf M$ is a b-nef b-Cartier divisor, $B+\mathbf M_X$ is big, and $\alpha=K_X+B+\mathbf M _X$ a nef $\R$-divisor. If $\alpha$ is not big, then there is a curve $C\subset X$ such that $\alpha\cdot C=0$.
% \end{proposition}
% \begin{proof} Let $\nu :X'\to X$ be a $\Q$-factorial dlt model such that $\nu$ is a projective morphism (see Theorem \ref{t-dltmodel}), and let $\alpha ' = \nu^* \alpha$. 
% By \cite[Corollary 1.3]{VP21} if $X'$ is not projective, then there is a curve $C$ such that $[-C]\in \overline {\rm NE}(X')$. But then $\alpha '\cdot C=0$ and so $\nu_*C =0$. It follows that $[-C]\in \overline {\rm NE}(X'/X)$\footnote{YL: Should we comment this proof?} which is impossible as $\nu $ is projective. Therefore, $X'$ is projective and then by \cite[Theorem 4.1]{DH24}, for every general point $x\in X'$ there is a rational curve $\Gamma _x$ with $\alpha'\cdot \Gamma _x=0$. It follows that $\nu _*\Gamma _x\ne 0$ and $\alpha\cdot \nu _*\Gamma _x=0$ and the proposition is proven.
% \end{proof}
\begin{theorem} 
\label{t-raynonbig}%Assume Conjecture \ref{c-BDPP13} in dimension $\leq n-1$.  
Let $X$ be a compact complex analytic variety in Fujiki's class $\mathcal C$ such that $(X,B+\bbeta)$ is generalized klt, $B+\bbeta _X$ is big, and $\alpha:=[K_X+B+\bbeta _X]$ is nef but not big. Then $X$ is covered by $\alpha $-trivial  rational curves. 
\end{theorem}
\begin{proof} Replacing $X$ by a $\Q$-factorialization (cf. Theorem \ref{t-3}), we may assume that $X$ is strongly $\Q$-factorial.
%Let $\nu :X'\to X$ be a resolution such that $X'$ is K\"ahler and there is an effective exceptional $\R$-divisor $F$ such that $-F$ is $\nu$-ample.

By Lemma \ref{l-kltbig} there is a generalized klt pair $(X,G+\ggamma)$ and a K\"ahler form $\omega'$ on $X'$ such that $K_{X}+B+\bbeta _{X}\equiv K_{X}+G+\ggamma _{X}$ and $\ggamma _{X'}-\omega'$ is nef.

%Since $B+\bbeta _X$ is  big, and $\alpha:=[K_X+B+\bbeta _X]$ is not big, then $K_X$ is not pseudo-effective.For any $0<\epsilon \ll 1$, $\omega ':=\nu ^*\omega -\epsilon E$ is K\"ahler and $(X',G'+\epsilon E)$ is sub-klt where $K_{X'}+G'+\ggamma _{X'}=\nu ^*(K_X+G+\ggamma_X)$. Replacing $E$ by $\epsilon E$ we may assume that $\epsilon =1$. We also write $G'=(G')^{\geq 0}+(G')^{\leq 0}$ where $(G')^{\geq 0},(G')^{\leq 0}$ are effective with no common components.

Let $X\dasharrow Z$ be the MRC fibration and $\nu :X'\to X$ be a resolution such that $X'$ is K\"ahler and $f':X'\to Z$ is a morphism. We may assume that $\bbeta$ and $\ggamma$ descend to $X'$. %that there is an effective $\nu$-exceptional $\R$-divisor $E$ such that $-E$ is $\nu$-ample. 
Note that $K_X$ is not pseudo-effective and hence $X$ is uniruled, $\dim X>\dim Z$. We may assume that $Z$ is smooth and K\"ahler. We write \[K_{X'}+G'+\ggamma _{X'}=\nu ^*(K_X+G+\ggamma _X)\equiv \nu ^*(K_X+B+\bbeta _X)\] and $G'=(G')^{\geq 0}-(G')^{\leq 0}$ where $(G')^{\geq 0},(G')^{\leq 0}$ are effective with no common components and let $\alpha' = \nu^* \alpha$.  Note that a general fiber $F$ of $f'$ is smooth and rationally connected, hence projective.  We also remark that $f:X\dasharrow Z$ is quasi-holomorphic and hence a holomorphic map over an open subset of $Z$. In particular, for general $z\in Z$, $\nu _z:X'_z\to X_z$ is a morphism between general fibers of $f'$ and $f$, and so if $E$ is $\nu$-exceptional, then $E|_F$  is $\nu |_F$-exceptional. 

We claim that $K_{X'/Z}+t\alpha '+(1-\delta) \omega ' $ is not pseudo-effective for any $0<\delta\leq 1$ and $t>0$. By \cite{Ou25}, $K_Z$ is pseudo-effective and so it suffices to show that $K_{X'}+t\alpha '+(1-\delta) \omega ' $ is not pseudo-effective.
If this is not the case, then 
\[K_{X'}+t\alpha'+(G')^{\geq 0}+(1-\delta) \omega '\equiv (t+1)\alpha '+(G')^{\leq 0}-(\delta\omega '+\gamma _{X'}-\omega ')\] is pseudo-effective and so $(t+1)\alpha '+(G')^{\leq 0}$ is big. Since $(G')^{\leq 0}$ is $\nu$-exceptional, this implies that $(t+1)\alpha$ is big, which is impossible.

By \cite[Theorem 5.2]{CH20}, it follows that $K_F+t\alpha _F+(1-\delta) \omega_F:=(K_{X'/Z}+t\alpha '+(1-\delta) \omega ' )|_F$  is not pseudo-effective and in particular $\alpha |_F$ is not big. Let \[K_{X'}+B_{X'}+\bbeta _{X'}=\nu ^*(K_X+B+\bbeta _X),\] and $\bbeta _F=\bbeta _{X'}|_F$, and let $B_F^{\ge 0}=(B_{X'}^{\ge 0}+\epsilon {\rm Ex}(\nu))|_F$, $B_F^{\le 0}=(B_{X'}^{\le 0}+\epsilon {\rm Ex}(\nu))|_F$ for $0<\epsilon \ll 1$. Since $F$ is projective and $h^2(\OO _F)=0$, then $\bbeta |_F\equiv  N$ where $ N$ is an $\R$-Cartier divisor (cf. \cite[Lemma 2.29]{DHY23}) and $B_F^{\geq 0}+N$ is big. Similarly $\alpha|_F \equiv L$ where $L$ is a nef $\R$-divisor. Then   \[L+B^{\leq 0}_F\equiv K_F+B_F^{\geq 0}+N\equiv K_F +\Delta _F\] where $(F,\Delta _F)$ is klt and $\Delta _F$ is big. 
By \cite{BCHM10}, we may run the corresponding MMP
which terminates with a good minimal model $\psi: F\dasharrow F'$ so that $K_{F'}+\Delta _{F'}=\psi _*(K_F+\Delta _F)$ is semi-ample. We note that since
$B_F^{\leq 0}$ is $\nu |_F$-exceptional and $L\equiv \alpha |_F$ is pulled back from $\nu (F)$, then the negative part of the divisorial Zariski decomposition
${N_{\sigma}}(L+B_F^{\leq 0})=B_F^{\leq 0}$ and so ${\rm Supp}(B_F^{\leq 0})$ is the set of divisors contracted by $\psi$. In particular $L'=\psi_*L\equiv K_{F'}+\Delta _{F'}$ is semi-ample. 

Let $\eta:F'\to \bar F$ be the corresponding morphism, then $K_{F'}+\Delta _{F'}=\eta ^*A _{\bar F}$ where $A _{\bar F}$ is ample. Since $K_{F}+\Delta _F$ is not big, then $F'$ is covered by $L'$-trivial curves.
We may assume that these curves are given by intersection of general very ample divisors with fibers of $\eta$, in particular they avoid any given codimension 2 closed subset of $F'$.
Since $B_F^{\leq 0}$ is $\psi$-exceptional, then $F$ is covered by $L$-trivial curves. 

Let $C$ be a very general such curve, then $F$ is smooth along $C$, $L\cdot C=0$, $ B_F^{\leq 0}\cdot C=0$, $\Delta _F\cdot C >0$ (as $\Delta _F$ is big) and so $K_F\cdot C<0$. By \cite[Theorem II.5.8]{Kol96}, through every point $x\in C$ there is a rational curve $D_x\subset F$ such that $L\cdot D_x=0$. Therefore $F$ is covered by $\alpha$-trivial rational curves and hence so is $X$.
% \sout{Let $\Upsilon$ be a face of $\overline{\rm NE}(F)$ defined by the $\alpha _F$-trivial $K_F+(1-\delta) \omega_F$ negative curves. Let $\eta:F\to \bar F $ be the corresponding contraction (which exists by} \cite[Theorem 1.2]{DH24}), \sout{then $\alpha _F=\eta ^*\alpha_{\bar F}$ where $\alpha_{\bar F}$ is K\"ahler. If $\eta $ is birational, then $\alpha _F $ is big which is a contradiction.
% Therefore $\eta$ is of fiber type and hence $F$ is covered by a family of $\alpha_F$-trivial rational curves $C_t$.} %We may assume that $0>K_F\cdot C_t\geq -2\dim F$ and hence $(1-\delta)\omega _F\cdot C_t<2\dim F$. But then these curves have bounded degree with respect to $\omega$ and so they must belong to finitely many numerical classes. Therefore, we may assume that  $\omega _F\cdot C_t\leq 2\dim F$.  But then $X$ is also covered by $\alpha$-trivial rational curves such that $0<-(K_X+B+\bbeta _X)\cdot C=\omega \cdot C\leq 2\dim X$.
\end{proof}

    % \begin{remark}
    % Note that the condition $B+\bbeta_X$ being big in the theorem above is necessary.
    % Consider in fact a complex torus $X$ with $B = 0$, $\bbeta = 0$, then
    % $\alpha:= [K_X+B+\bbeta_X]\equiv 0$ % is nef but not big, and $B+\bbeta_X = 0$ is not big either. In this case 
    % but $X$ contains no rational curves.%, hence no $\alpha$-trivial rational curves.
    % \end{remark}

As a consequence of the proof of Theorem~\ref{t-raynonbig}, on a compact K\"ahler variety we also obtain a length bound of these $\alpha$-trivial rational curves.
\begin{corollary}\label{c-Kahleralphatrivial}
Let $X$ be a compact strongly $\mathbb{Q}$-factorial K\"ahler variety,  $(X,B+ \bbeta)$ a generalized klt pair, and  a K\"ahler form $\omega$ on $X$. Assume that $\alpha:=[K_X+B+\boldsymbol\beta_X+\omega]$ is nef but not big. Then $X$ is covered by $\alpha$-trivial rational curves $\Gamma$ such that $$ 0<-(K_X+B+\boldsymbol\beta_X)\cdot\Gamma=\omega\cdot\Gamma\le 2\dim X .$$
\end{corollary}
\begin{proof}
Since $\omega$ is K\"ahler, $B+\bbeta_X+\omega$ is big. Applying Theorem~\ref{t-raynonbig} to the generalized pair $(X,B+(\bbeta+\overline{\omega}))$, we see that $X$ is covered by $\alpha$-trivial rational curves, so it remains to prove the length estimate. To this end we use the construction in the proof of Theorem~\ref{t-raynonbig}. Let $\nu\colon X'\to X$ be the resolution used there and $E$ an effective $\nu$-exceptional divisor such that $\nu^*\omega-\epsilon E$ is K\"ahler for $0<\epsilon\ll 1$. Let $F$ be a general fiber of the MRC fibration $f'\colon X'\to Z$, then $F$ is smooth, projective and rationally connected. Since $h^2(\OO_F)=0$, there are $\R$-divisors $L$, $N$ and $A$ on $F$ such that
\[
\nu^*\alpha|_F\equiv L,\qquad \bbeta_{X'}|_F\equiv N,\qquad (\nu^*\omega)|_F\equiv A,\qquad K_F+B_F+N+A\equiv L,
\]
where $B_F=B_{X'}|_F$ (cf. \cite[Lemma 2.29]{DHY23}). Here $L$ and $N$ are nef, and $A$ is nef and big, since $\nu|_F$ is birational onto its image and $\omega$ is K\"ahler. Moreover, as in the proof of Theorem~\ref{t-raynonbig}, there is a covering family of curves on $F$ whose very general member $C$ satisfies $L\cdot C=0$ and $C\cap {\rm Supp}(B_F^{\leq 0})=\emptyset$.

We first show that $-K_F\cdot C\geq A\cdot C>0$. Since $K_F+B_F^{\geq 0}-B_F^{\leq 0}+N+A\equiv L$, and $L\cdot C=B_F^{\leq 0}\cdot C=0$, we obtain
\[
-K_F\cdot C=B_F^{\geq 0}\cdot C+N\cdot C+A\cdot C.
\]
Since $C$ is very general, it is not contained in ${\rm Supp}(B_F^{\geq 0})$, so $B_F^{\geq 0}\cdot C\geq 0$, and since $N$ is nef, $N\cdot C\geq 0$. Hence $-K_F\cdot C\geq A\cdot C$. Finally, $A\cdot C>0$, because $A$ is big and $C$ is a very general member of a covering family.

Given an arbitrary point $x\in C$ and an integer $m\geq 1$, by \cite[Theorem II.5.8]{Kol96} there is a rational curve $C_m\subset F$ passing through $x$ such that
\[
(mL+A)\cdot C_m\leq 2\dim F\cdot\frac{(mL+A)\cdot C}{-K_F\cdot C}=2\dim F\cdot\frac{A\cdot C}{-K_F\cdot C}\leq 2\dim F .
\]
Therefore, $
L\cdot C_m\leq \frac{2\dim F}{m}$ and $A\cdot C_m\leq 2\dim F$, as $L$ and $A$ are nef.

We claim that $L\cdot C_m=0$ for $m\gg 0$. Since $x$ is a general point of $X'$, the curve $C_m$ is not contained in ${\rm Supp}(E)$, so $E\cdot C_m\geq 0$ and
\[
(\nu^*\omega-\epsilon E)\cdot C_m\leq \nu^*\omega\cdot C_m=A\cdot C_m\leq 2\dim F .
\]
Thus, the curves $C_m$ have bounded degree with respect to the K\"ahler form $\nu^*\omega-\epsilon E$ on $X'$, and by Bishop's theorem (cf. \cite[Theorem 1.1]{Lieb78}) they belong to finitely many numerical classes. Therefore $L\cdot C_m=\nu^*\alpha\cdot C_m$ takes only finitely many values. These values are non-negative and at most $2\dim F/m$, so $L\cdot C_m=0$ for $m\gg 0$.

Fix such an $m$ and let $\Gamma:=\nu_*C_m$. Since $x$ is general, $C_m$ is not contracted by $\nu$, so $\Gamma$ is a rational curve passing through $\nu(x)$. By the projection formula, $\alpha\cdot\Gamma=L\cdot C_m=0$ and $\omega\cdot\Gamma=A\cdot C_m$. Since $\alpha=K_X+B+\bbeta_X+\omega$ and $\omega$ is K\"ahler, we conclude that
$$
0<-(K_X+B+\bbeta_X)\cdot\Gamma=\omega\cdot\Gamma\leq 2\dim F\leq 2\dim X .
$$\end{proof}

Combining Theorem \ref{t-3} and Theorem \ref{t-raynonbig} with \cite{HP24}, we obtain the following generalization of \cite[Theorem 0.5]{HP24}.
\begin{corollary}\label{c-conetheorem} Let $(X,B+\bbeta )$ be a  compact K\"ahler generalized klt pair. Then there are at most countably many rational curves $\{\Gamma _i\}_{i\in I}$ such that $0<-(K_X+B+\bbeta _X)\cdot \Gamma _i\leq 2\dim X$ for all $i\in I$ and  \[\overline{\rm NA}(X)=\overline{\rm NA}(X)_{(K_X+B+\bbeta _X) \geq 0}+\sum _{i\in I}\mathbb R ^+[\Gamma _i].\]
Moreover, if $B+\bbeta _X$ or $K_X+B+\bbeta _X$ is big, then $I$ is finite.
\end{corollary}
\begin{proof} By Theorem \ref{t-3} there exists
a small birational map $\nu : X' \to X$ such that $X'$ is strongly $\Q$-factorial. Let $K_{X'}+B'+\bbeta _{X'}=\nu ^*(K_X+B+\bbeta_X)$, then $(X',B'+\bbeta)$ is generalized klt. Arguing as in the proof of \cite[Corollary 5.3]{DHP24}, we may replace $(X,B+\bbeta)$ by $(X',B'+\bbeta)$, and so we may assume that $X$ is strongly $\Q$-factorial.
By the proof of \cite[Proposition 3.1, Corollary 3.2]{HP24}, it suffices to show that if $\omega$ is any K\"ahler form such that $\alpha=K_X+B+\bbeta _X+\omega$ is nef but not K\"ahler, then there is
an $\alpha$-trivial rational curve $C$ such that $0<-(K_X+B+\bbeta _X)\cdot C\leq 2\dim X$.
By \cite[Proposition 3.1]{HP24}, this holds if $K_{X}+B+\bbeta _{X}+\omega$ is big. If instead $K_{X}+B+\bbeta _{X}+\omega$ is not big, then it holds by Corollary \ref{c-Kahleralphatrivial}. 
\end{proof}
\section{Proof of Theorem \ref{t-contK}}

This section is devoted to the proof of Theorem \ref{t-contK}. The technical heart of the proof is Theorem \ref{t-ray}, a K\"ahlerness criterion formulated in terms of $\alpha$-trivial rational curves.

\begin{definition}[Null locus and restricted non-K\"ahler locus]
Let $X$ be a compact normal variety in Fujiki's class $\mathcal{C}$, and let $\alpha \in H^{1,1}_{\rm{BC}}(X)$ be a nef and big class. The null locus of $\alpha$ is defined to be
$$\text{Null}(\alpha) = \bigcup_{\int_{V} \alpha^{\dim V} = 0} V,$$
where $V \subset X$ are positive dimensional irreducible analytic subsets of $X$. The restricted non-K\"ahler locus is defined to be
$$
E_{n K}^{a s}(\alpha):=\bigcap_{T \in \alpha} E_{+}(T),
$$
where $T$ above ranges over K\"ahler currents in $\alpha$ with weak analytic singularities, and $E_{+}(T) \subset X$ is the analytic subset of $X$ consisting of points at which the Lelong numbers of $T$ are strictly positive.
\end{definition}

The null locus and the restricted non-K\"ahler locus of a nef and big class coincide on a compact normal complex analytic variety in Fujiki's class $\mathcal{C}$.

\begin{theorem}\label{t-null} Let $X$ be a compact normal variety in Fujiki's class $\mathcal C$, and let $\alpha$ be a smooth
(1,1)-form, which is locally $\partial\bar\partial$-exact and such that the corresponding class is nef
and big. Then $E^{as}
_{nK} (\alpha ) = {\rm Null}(\alpha )$. In particular, the set ${\rm Null}(\alpha )$ is analytic.
\end{theorem}
\begin{proof} The smooth case is proven in \cite{CT15}.
The general case is proven in \cite[Theorem 4.21]{HP24}.
\end{proof}

\begin{theorem} 
\label{t-ray}%Assume Conjecture \ref{c-BDPP13} in dimension $\leq n-1$.  
Let $X$ be a compact complex analytic variety in Fujiki's class $\mathcal C$ of dimension $n$ such that $(X,B+\bbeta)$ is generalized klt, $\alpha:=[K_X+B+\bbeta _X]$ is nef and big but not K\"ahler. Suppose that $\alpha$ is NQC (see \cite[Section 2.10]{HX26} for the definition), then there is an $\alpha $-trivial rational curve $C\subset X$.
\end{theorem}
\begin{proof} 
 %Replacing $\bbeta$ by $\bbeta+\epsilon \bar \alpha$ we may assume that $\bbeta _X$ is big.
 By Theorem \ref{t-null} (see also \cite[Theorem 1.1]{CT15} and \cite[Theorem 3.17]{Bou04} in the smooth case) the restricted non-K\"ahler locus
${ E_{nK}^{as}} (\alpha)$ coincides with the null-locus  ${\rm Null}(\alpha)$, and 
there exists a K\"ahler current $\eta$ with weak analytic singularities in the class $\alpha$ such that
the Lelong set coincides with ${ E_{nK}^{as}} (\alpha)$. 
Since $\alpha $ is not K\"ahler, then ${\rm Null}(\alpha)$ has a positive dimensional component.
We may pick a log resolution $\nu :X'\to X$ such that $K_{X'}+B_{X'}+\bbeta _{X'}=\nu ^*(K_X+B+\bbeta _X)$ where $\bbeta _{X'}$ is nef and
$\nu ^*\eta=\eta' +E$ where $E$ is an effective $\R$-divisor, $\eta'$ is a nef current such that $\eta'\geq \nu ^*g$ where $g$ is a positive definite real (1,1)-form.

Let $Z$ be a maximal dimensional irreducible component of ${\rm Null}(\alpha)$ and $c$ be the log canonical threshold of $(X,B+\bbeta)$ with respect to   $\overline{\eta'}+\nu_*E$ on a neighborhood of general points of $Z$. We can assume that  $E+B_{X'}$ has simple normal crossings, then $Z$ is an irreducible component of $\nu ((B_{X'}+cE)^{=1})$ and $Z$ is not contained in $\nu ((B_{X'}+cE)^{>1})$. %\sout{We may also assume that $\nu$ is given by a sequence of blow ups along smooth centers and in particular there is an effective $\nu$-exceptional $\R$-divisor $F$ such that $-F$ is relatively ample.} 
We may write $\nu ^*\alpha \equiv \omega '+G$ where $\omega '$ is K\"ahler and $G\geq 0$ is an effective $\R$-divisor. We may replace $\omega '$ and $G$ by a small perturbation so that the coefficients of $G$ are very general.
We may pick $0<\epsilon \ll 1$ such that if \[c'={\rm sup}\{ t|Z\not\subset \nu ((B_{X'}+t((1-\epsilon)E+\epsilon G))^{\geq 1}) \}\]
then $Z$ is the image of a unique irreducible component $S$ of $(B_{X'}+c'((1-\epsilon)E+\epsilon G))^{=1}$ and is not contained in $\nu(B_{X'}+c'((1-\epsilon)E+\epsilon G))^{>1}$.

We let $\ggamma =\bbeta+c'\left( \overline{(1-\epsilon)\eta'+\epsilon \omega ' }\right)$ and $\Gamma '=B_{X'}+c'((1-\epsilon)E+\epsilon G)$. Then we have a generalized pair $(X,\Gamma +\ggamma)$ where $\Gamma=\nu _*\Gamma'$ and  
 $K_X+\Gamma+\ggamma _X\equiv (1+c')\alpha$. Let $Z^\nu\to Z$ be the normalization and $g:Z'\to Z^\nu$ be a resolution, $\alpha_{Z^\nu}=\alpha |_{Z^\nu}$ and $\alpha _{Z'}=\alpha |_{Z'}$. We may assume that $Z'$ is K\"ahler and $S\to Z$ factors through $f:S\to Z'$. We display the following diagram for the reader’s convenience:
\begin{center}
\begin{tikzcd}
S \arrow[d, hook] \arrow[r, "f"] & Z' \arrow[r, "g"] & Z^\nu \arrow[r] & Z \arrow[d, hook] \\
X' \arrow[rrr, "\nu"'] & & & X.
\end{tikzcd}
\end{center}
We may assume that $X'\to X$ factors through $\mu :X'\to \bar X$ where $\rho:\bar X\to X$ is a dlt model (in the sense of Theorem \ref{t-dltmodel}) for $(X,\Gamma+\ggamma  )$. In particular, \[K_{\bar X}+\bar \Gamma+\ggamma_{\bar X}=\rho ^*(K_X+\Gamma +\ggamma _X)\equiv (1+c')\bar \alpha\] where $\bar \alpha =\rho ^*\alpha$, $\bar \Gamma =\mu _*\Gamma' \geq 0$, and $(\bar X,\rho ^{-1}_*(\Gamma\wedge {\rm Supp}(\Gamma)) +{\rm Ex}(\rho))$ is dlt. By our construction, we have that if $\bar S:=\mu _* S$, then $\bar S\ne 0$ is normal, and hence $\bar f:\bar S\to Z^\nu$ is a morphism.
 We write
\[K_S+\Delta _S+\ddelta _S=(K_{X'}+\Gamma '+\ggamma_{X'} )|_{S}\]
where $\Delta _S=(\Gamma '-S)|_S$ and $\ddelta _S=(\ggamma _{X'})|_{S}$. We let $K_{\bar S}+\Delta _{\bar S}+\ddelta _{\bar S}=\mu _*(K_S+\Delta _S+\ddelta _S)$.
% {\color{blue} Since $K_S+\Delta _S+\ddelta _S\equiv (1+c')f^*\alpha_{Z'}$ and $(S,\Delta _S+\ddelta _S)$ is sub-klt over general points of $Z$, then 
% \[K_S+\Delta _S+\ddelta _S=f^*(K_{Z'}+\Delta _{Z'}+\rrho _{Z'})\]
% where $\Delta _{Z'}$ is the boundary part. By \cite[Proposition 1.16]{HP24} we may assume that $\rrho _{Z'}$ descends to $Z'$
Let  $\tilde \Delta _S=\Delta _S-\Delta _S^{\geq 1}+\Delta _S^{<0}$, then $(S,\tilde \Delta _S)$ is klt, and
\[ K_S+\tilde \Delta _S+\ddelta _S\equiv (1+c')f^*\alpha _{Z'} -\Delta _S^{\geq 1}+\Delta _S^{<0}.\]
Note that $\Delta _S^{\geq 1}$ does not dominate $Z'$, $\ddelta _S$ is K\"ahler, and general fibers of $S\to Z$ are rationally connected (cf. Lemma \ref{l-pltrcc}).  By \cite[Theorem 3.1]{CH24}, $f$ is a projective morphism. %Thus $-(K_S+\Delta _S)\equiv _{Z'}\ddelta _S$ is ample over $Z'$} and hence $f$ is a projective morphism. 

Since $\Delta _S^{\geq 1}$ does not dominate $Z'$, then $K_S+\tilde \Delta _S+\ddelta _S$ is pseudo-effective over $Z'$ and since $ \ddelta _S\equiv _{Z'}-(K_S+\tilde \Delta _S)-\Delta _S^{\geq 1}+\Delta _S^{<0}$, then by Proposition \ref{p-relmmp} there is a good minimal model $\phi:S\dasharrow S'$ for $K_S+\tilde \Delta _{S}+\ddelta _{S}$ over $Z'$, and therefore, we may assume that $K_{S'}+\tilde \Delta _{S'}+\ddelta _{S'}$ is nef over $Z'$ and there is a log canonical model $f'':S'\to Z''$ over $Z'$. For general $z\in Z'$, let $S_z$ be the fiber of $f:S\to Z'$ over $z$ and $\bar S_z$ the fiber of $\bar f:\bar S\to Z^\nu$ over $g(z)$. Since $\Delta _S^{<0}|_{S_z}$ is $S_z\dasharrow \bar S_z$ exceptional, then $h:Z''\to Z'$ is birational and in particular $S_z$ is rationally connected (see Lemma \ref{l-pltrcc}). Then $K_{S'}+\tilde \Delta _{S'}+\ddelta _{S'}=(f'')^*\lambda ''$ where $\lambda ''\in H^{1,1}_{\rm BC}(Z'')$ is K\"ahler over $Z'$. Let $f':S'\to Z'$ be the induced morphism.  %\[K_{S'}+\Gamma _{S'}+\ddelta _{S'}=(f'')^*(K_{Z''}+\Gamma _{Z''}+ \rrho _{S''}) .\]
\begin{center}

\begin{tikzcd}
	{S'} && {Z''} \\
	& {Z'}
	\arrow["{f''}", from=1-1, to=1-3]
	\arrow["{f'}"', from=1-1, to=2-2]
	\arrow[" ", from=1-3, to=2-2]

\end{tikzcd}
\end{center}
We claim that $K_{S'/Z'}+\tilde \Delta _{S'}+\ddelta _{S'}-(f')^*\omega _{Z'}$ is pseudo-effective for some K\"ahler form $\omega _{Z'}$. To see this, we first observe that $\ddelta _S$ is K\"ahler, so we may assume that $\ddelta _S- f^*\omega _{Z'}$ is a K\"ahler class. Let $\tilde \ddelta $ be $\ddelta -\overline{f^*\omega _{Z'}}$, it suffices to show that $K_{S'/Z'}+\tilde \Delta _{S'}+\tilde \ddelta _{S'}$ is pseudo-effective.
Let $Z'''\to Z''$ be a resolution and $\pi: S''\to S'$ a resolution such that the induced map $f''':S''\to Z'''$ is a morphism.  We write $K_{S''}+\tilde \Delta _{S''}+\tilde\ddelta _{S''}=\pi^*(K_{S'}+\tilde \Delta _{S'}+\tilde\ddelta _{S'})\equiv _{Z''}0$. Then \[K_{S''}+\tilde \Delta _{S''}^{\geq 0}+\tilde\ddelta _{S''}\equiv _{Z'''}\tilde{\Delta}_{S''}^{\leq 0}.\]
By \cite[Theorem 2.2]{HP24} (see also \cite[Theorem 2.47]{HX26} for the $\R$-divisor case) 
$K_{S''/Z'''}+\tilde \Delta _{S''}^{\geq 0}+\tilde\ddelta _{S''}$ is pseudo-effective. Since $K_{Z'''/Z'}\geq 0$, then $K_{S''/Z'}+\tilde \Delta _{S''}^{\geq 0}+\tilde\ddelta _{S''}$ is pseudo-effective. Since $\tilde \Delta _{S''}^{\leq 0}$ is $\pi$-exceptional, 
$K_{S'/Z'}+\tilde \Delta _{S'}+\tilde\ddelta _{S'}=\pi _*(K_{S''/Z'}+\tilde \Delta _{S''}^{\geq 0}+\tilde\ddelta _{S''})$ is pseudo-effective.

If $Z'$ is not uniruled, then $K_{Z'}$ is pseudo-effective (cf. \cite{Ou25}) and so  $K_{S'}+\tilde \Delta _{S'}+ \ddelta _{S'}-(f')^*\omega _{Z'}$ is pseudo-effective and hence so is $K_{S}+\tilde \Delta _{S}+\Delta _S^{\geq 1}+ \ddelta _{S}-f^*\omega _{Z'}$ (since $\phi$ is a $K_{S}+\tilde \Delta _{S}+ \ddelta _{S}$ MMP over $Z'$ and $\Delta _S^{\geq 1}$ is effective).

Pick $\eta _{Z^\nu}$ a big $(1,1) $ class on $Z^\nu$ such that $\omega _{Z'}-g^*\eta _{Z^\nu}$ is pseudo-effective, then \[K_{\bar S}+\Delta _{\bar S} +\ddelta _{\bar S}-\bar f ^*\eta _{Z^\nu}=(\mu |_S)_*( K_{S}+\Delta _{S}^{\geq 0}+ \ddelta _{ S}-f^*g^*\eta _{Z^\nu})\]
is pseudo-effective. Since \[K_{\bar S}+\Delta _{\bar S} +\ddelta _{\bar S}-\bar f ^*\eta _{Z^\nu}=\bar{f}^*((1+c')\alpha_{Z^\nu}- \eta _{Z^\nu})\] it follows that $(1+c')\alpha _{Z^\nu}- \eta _{Z^\nu}$ is pseudo-effective and hence $\alpha _{Z^\nu}$ is big which is impossible.

Therefore, $Z'$ is uniruled. Let $\psi:Z'\dasharrow Y$ be the MRC fibration. We may assume that $\psi$ is a morphism with general fiber $F$. Note that $F$ is smooth and rationally connected, hence projective. By an argument similar to the one above, one sees that $K_{Z'}+t\alpha_{Z'}+\omega _{Z'}$ is not pseudo-effective for any $t>0$ either. Since $Y$ is not uniruled, then $K_Y$ is pseudo-effective and so $K_{Z'/Y}+t\alpha_{Z'}+\omega _{Z'}$ is not pseudo-effective. By \cite[Theorem 5.2]{CH20}, $K_F+t\alpha _F+\omega _{Z'}|_F$ is not pseudo-effective and so $\alpha _F:=\alpha |_F$ is nef but not big. Notice that there are $\R$-divisors $H$ and $A$ numerically equivalent to $\alpha_F$ and $\omega_{Z'}|_F$ on $F$. 

Since $\alpha$ is NQC, it is easy to check that $H$ is NQC in the usual sense. Then for any $t\gg 0$, we can run an $H$-trivial, $(K_F+tH+A)$-MMP (by \cite[Lemma 3.22]{HL22}) that terminates with a Mori fiber space. In particular, $F$ is covered by a family of $H$-trivial rational curves. Therefore $Z$ is covered by $\alpha$-trivial rational curves as well.
\end{proof}

\begin{theorem} 
\label{t-raynklt} 
Let $X$ be a compact complex analytic variety in Fujiki's class $\mathcal C$ of dimension $n$ such that $(X,B+\bbeta)$ is a generalized pair with non-klt locus $W$, $\alpha:=[K_X+B+\bbeta _X]$ is nef NQC and big but not K\"ahler, and $\alpha |_{W}$ is K\"ahler. Then there is an $\alpha $-trivial  rational curve $C$ not contained in $W$.
\end{theorem}
\begin{proof} The proof is the same as above. Suppose that ${\rm Null}(\alpha)\ne \emptyset$. Since $\alpha |_{W}$ is K\"ahler, then no component of ${\rm Null}(\alpha)$ is contained in $ W$. Let $Z$ be a maximal dimensional irreducible component of ${\rm Null}(\alpha)$, then the proof of Theorem \ref{t-ray} applies. 
\end{proof}
\begin{lemma}
\label{l-nefpb} Let $X$ be a  compact complex analytic variety in Fujiki's class $\mathcal C$ and $\nu :X'\to X$ a birational morphism such that $X'$ is K\"ahler and $X$ has rational singularities. If  $\alpha \in H^{1,1}_{\rm BC}(X)$ and $\nu ^*\alpha$ is nef, then $\alpha$ is nef. In particular, we have $\nu_*\overline{\rm NA}(X')=\overline{\rm NA}(X)$.

\end{lemma} 
\begin{proof} By \cite[Remark 2.37]{DHP24}, $\alpha$ is nef if and only if $\alpha|_{Z^n}$ is a pseudo-effective class for all irreducible analytic
subvarieties $Z \subset X$ with normalization $Z^n\to Z$.
Let $Z'$ be the normalization of a component of $\nu^{-1}(Z)$ that dominates $Z$, $d=\dim(Z'/Z)$, and $\mu:Z'\to Z^n$ the induced map. Fix $\omega$ a K\"ahler form on $Z'$, then $\psi:=\int _{Z'} (\nu ^*\alpha\wedge \omega ^d \wedge\ldots )$ is a $(d+1,d+1)$ current on $Z'$ such that
$\mu _* \psi =C \alpha |_{Z^n}$
 where $C=\int _F \omega ^d$ and $F$ is a general fiber of $Z'\to {Z^n}$. To see this, let $\gamma$ be any smooth $(\dim Z-1,\dim Z-1)$ form on ${Z^n}$, then  integrating along the fibers first gives
\[\int _{Z'} \mu ^*\alpha\wedge \omega ^d \wedge \mu ^*\gamma =C \int _{Z^n}\alpha \wedge \gamma .\]
Since  $\mu ^*\alpha$ is nef then if $\gamma\geq 0$, the corresponding integral is $\int _{Z^n}\alpha \wedge \gamma\geq 0$.
Hence $\alpha|_{Z^n}=\frac 1 C \mu _* \psi $  is a positive current and hence  $\alpha |_{Z^n}$ is pseudo-effective.

It is clear that $\nu_*\overline{\rm NA}(X')\subseteq\overline{\rm NA}(X)$. If they are not equal, then there is a class $\beta\in H^{1,1}_{\rm BC}(X)$ such that $\beta$ is non-negative on $\nu_*\overline{\rm NA}(X')$ but not on $\overline{\rm NA}(X)$. It follows from the projection formula that $\nu^*\beta$ is nef. By what we have seen above, $\beta$ is also nef, and thus $\beta$ is non-negative on $\overline{\rm NA}(X)$, which contradicts our choice of $\beta$. 
\end{proof}
%\begin{corollary} Let $(X,B+\bbeta)$ be a K\"ahler, compact generalized klt strongly $\Q$-factorial pair and $f:X\to Z$ be a flipping or divisorial contraction, then $Z$ is K\"ahler.\end{corollary}

We conclude this section with the proof of Theorem \ref{t-contK}.

\begin{proof}[Proof of Theorem \ref{t-contK} when $B+\bbeta_X$ is big] By Proposition \ref{p-nefsupp}, there is a $K_X+B+\bbeta _X$ negative extremal ray $\Gamma$, which is  cut out by a nef class $\alpha =K_X+B+\bbeta _X+\omega$ where $\omega $ is K\"ahler. By Theorem \ref{t-raynonbig}, if $\alpha$ is not big, then $X$ is covered by an $\alpha$-trivial family of curves $C_t$. Since $\alpha$ cuts out an extremal ray $\Gamma$ of $\overline{\rm NA}(X)$, then $[C_t]\in \Gamma$. This is impossible as the curves in $\Gamma$ cover a proper closed subset of $X$ (by definition of flipping or divisorial contraction associated to an extremal ray $\Gamma$). Therefore,  we may assume that $\alpha$ is big.

We first check that $\alpha =f^*\alpha _Z$ where $\alpha_Z\in H^{1,1}_{\rm BC}(Z)$ and $(Z,B_Z+\bbeta _Z+f_*\omega )$ is generalized klt where $B_Z=f_*B$. Since the question is local over $Z$, we may assume that $Z$ is relatively compact and Stein. 
Let $\nu :X'\to X$ be a resolution such that $f':X'\to Z$ is projective and $\bbeta $ descends to $X'$. Further shrinking $Z$, we may assume that $f'$ satisfies property (P). Write $K_{X'}+B'+\bbeta _{X'}=\nu ^*(K_X+B+\bbeta _X)$. 
Note that $X$ has rational singularities, since $(X,B+\bbeta)$ is generalized klt. Since $\alpha\cdot\Gamma=0$, the class $\nu^*\alpha$ is trivial on every $f'$-vertical curve, so
\[-(K_{X'}+B')=\bbeta_{X'}+\nu^*\omega-\nu^*\alpha\]
is an $f'$-nef $\R$-divisor, and it is $f'$-big since $f'$ is bimeromorphic. Note that $-(K_{X'}+B')$ is also $\nu$-nef and $\nu$-big. Write $B'=P-N$, where $P,N\geq0$ have no common components. The $\R$-divisor $-(K_{X'}+B')=\bbeta_{X'}+\nu^*\omega-\nu^*\alpha$ is nef and big over both $X$ and $Z$, so relative Kawamata--Viehweg vanishing gives
\[R^i\nu_*\OO_{X'}(\lceil N\rceil)=R^if'_*\OO_{X'}(\lceil N\rceil)=0\qquad(i>0).\]
As $\lceil N\rceil$ is effective and $\nu$-exceptional, $R\nu_*\OO_{X'}(\lceil N\rceil)\cong\OO_X\cong R\nu_*\OO_{X'}$, the latter because $X$ has rational singularities. Hence $R^if'_*\OO_{X'}\cong R^if_*\OO_X\cong R^if'_*\OO_{X'}(\lceil N\rceil)=0$ for $i>0$, and so $Z$ has rational singularities. By \cite[Lemma 2.6.(1)]{DH24}, $\alpha =f^*\alpha _Z$ where $\alpha_Z\in H^{1,1}_{\rm BC}(Z)$. 

By Lemma \ref{l-nefpb}, $\alpha _Z$ is nef, and by \cite[Lemma 2.45]{HX26} we know that $\alpha$ is NQC. We may assume $\alpha=\sum_ia_i\alpha_i$, where $\alpha_i\in f^*H^2(Z,\Q)$ is non-negative on $\overline{\rm NA}(X)$. Let $\alpha_i=f^*\alpha_{i,Z}$ with $\alpha_{i,Z}\in H^2(Z,\Q)$, then $\alpha_{i,Z}$ is non-negative on $\overline{\rm NA}(Z)$ by Lemma \ref{l-nefpb} again, hence $\alpha_Z$ is also NQC. Since $f$ contracts all $\alpha$-trivial curves, then $\alpha _Z\cdot C>0 $ for any curve $C\subset Z$. By Theorem \ref{t-ray}, $\alpha _Z$ is K\"ahler.
\end{proof}

\begin{proof}[Proof of the general case]
We reduce to the case where $B+\bbeta_X$ is big. Let $\omega$ be a K\"ahler form on $X$, and let $C$ be a curve contracted by $f$, so that $\Gamma=\R_{\geq 0}[C]$. Set
\[
a:=-(K_X+B+\bbeta_X)\cdot C>0,\qquad b:=\omega\cdot C>0,
\]
fix $0<\epsilon<a/b$, and let $\bbeta':=\bbeta+\epsilon\overline{\omega}$. Since $\overline{\omega}$ descends to $X$, the pair $(X,B+\bbeta')$ is generalized klt, and $B+\bbeta'_X=B+\bbeta_X+\epsilon\omega$ is big because $\omega$ is K\"ahler.

We claim that $f$ is also a $(K_X+B+\bbeta'_X)$-flipping (resp.\ divisorial) contraction associated to $\Gamma$. We only need to check that $-(K_X+B+\bbeta'_X)$ is K\"ahler over $Z$. Since $\rho_{\rm BC}(X/Z)=1$, we have $\omega\equiv_Z \frac{b}{a}\bigl(-(K_X+B+\bbeta_X)\bigr)$, and hence
\[
-(K_X+B+\bbeta'_X)\equiv_Z\Bigl(1-\frac{\epsilon b}{a}\Bigr)\bigl(-(K_X+B+\bbeta_X)\bigr).
\]
As $1-\epsilon b/a>0$ and $-(K_X+B+\bbeta_X)$ is K\"ahler over $Z$, the claim follows. 
\end{proof}

\begin{corollary}\label{c-kahlermmp} Let $(X,B+\bbeta)$ be a compact K\"ahler strongly $\Q$-factorial generalized klt pair. If $f:X\dasharrow X'$ is a flip or divisorial contraction, then $X'$ is strongly $\Q$-factorial and K\"ahler. \end{corollary}
\begin{proof} The strong $\Q$-factoriality of $X'$ is shown in \cite[Lemma 2.5]{DH25}. If $f$ is a divisorial contraction, then $X'$ is K\"ahler by Theorem \ref{t-contK}. If $f$ is a flip, then let $g:X\to Z$ be the flipping contraction and $g':X'\to Z$ the flipped contraction. By Theorem \ref{t-contK}, $Z$ is K\"ahler. Since the class of $K_{X'}+B'+\bbeta _{X'}=f_*(K_X+B+\bbeta _X)$ is relatively K\"ahler over $Z$, then $X'$ is K\"ahler.
\end{proof}

\section{Proof of Theorem \ref{c-1}}%\begin{theorem}\label{c-1} Let $(X,B+\bbeta)$ be a compact generalized klt pair in Fujiki's class $\mathcal C$. If $X$ is not K\"ahler then either \begin{enumerate} \item $X$ contains a rational curve such that $-[C]\in \overline{\rm NA}(X)$, or \item there is a small $\Q$-factorial K\"ahler modification $\mu:X^{\rm qf}\to X$. \end{enumerate} \end{theorem}%\begin{corollary} Let $X$ be a compact Fujiki class $\mathcal C$ variety such that $(X,B+\bbeta)$ is generalized klt $K_X+B+\bbeta _X$ is nef and $(K_X+B+\bbeta _X)\cdot C>0$ for any curve $C\subset X$, then $X$ is K\"ahler.\end{corollary}

\begin{proof}[Proof of Theorem \ref{c-1}]

Since $(X,B+\bbeta)$ is generalized klt, $X$ has rational singularities. Let $\nu :X'\to X$ be a projective resolution such that $X'$ is K\"ahler, and $\omega '$ a K\"ahler class on $X'$ which is very general in $H^{1,1}_{\rm BC}(X')$. Let $B'=\nu ^{-1}_*B+(1-\epsilon){\rm Ex}(\nu)$, then \[K_{X'}+B'+\bbeta _{X'}=\nu^*(K_X+B+\bbeta _X)+G\] where $G\geq 0$ and ${\rm Supp}(G)$ is the set of $\nu$-exceptional divisors. Then the negative part of the divisorial Zariski decomposition  $N_{\sigma}(K_{X'}+B'+\bbeta _{X'}/X)=G$ (see Proposition \ref{p-negZariski}),  and the support of $N_{\sigma}(K_{X'}+B'+\bbeta _{X'}+\epsilon \omega ' /X)$ is  ${\rm Ex}(\nu)$ for any $0<\epsilon \ll 1$. Replacing $\omega '$ by a multiple, we may assume that $K_{X'}+B'+\bbeta _{X'}+ \omega ' $ is K\"ahler. 
 We will now run the $K_{X'}+B'+\bbeta _{X'}$ MMP with scaling of $\omega '$ over $X$.  
 %Let $X=\cup U_i$ be a finite cover by relatively compact Stein open subsets such that $\nu _i:U'_i\to U_i$ satisfies property (P) where $U_i'=\nu ^{-1}(U_i)$. 
 Suppose that $K_{X'}+B'+\bbeta _{X'}+t\omega '$ is nef over $X$, but $K_{X'}+B'+\bbeta _{X'}+t'\omega '$ is not nef over $X$ for any $t'<t$. 
 
 If $t=0$, then $K_{X'}+B'+\bbeta _{X'}$ is nef over $X$ and hence $G=N_{\sigma}(K_{X'}+B'+\bbeta _{X'}/X)=0$ so ${\rm Ex}(\nu)$ contains no divisors and in particular $\nu$ is a small birational morphism such that $X'$ is K\"ahler. %By Theorem \ref{t-3} there exists a small strongly $\Q$-factorial modification $\mu: X^{\rm qf}\to X'$ where $\mu$ is projective, and hence $X^{\rm qf}$ is K\"ahler and $X^{\rm qf}\to X$ is also a small strongly $\Q$-factorial modification.\Xieinline{Why are we doing this, isn't $X'$ already smooth in the beginning? Maybe we should put this argument in the end instead. YL: Yes I agree, it's already Q-factorial after the resolution.}
 Thus the claim follows from Theorem \ref{t-3}.
 
We will henceforth assume that $t>0$ and we let $\psi :X'\to \bar X$ be the log canonical model for $K_{X'}+B'+\bbeta _{X'}+t\omega '$ over $X$ (see Proposition \ref{c-relmmp}), and  $F\subset \overline{\rm NA}(X')$ be the cone spanned by $\psi$ vertical curves. Since $X'$ is K\"ahler, $\psi$ is a non-trivial birational map, and in particular Moishezon so that $F\ne 0$.  Since $\omega '\in H^{1,1}_{\rm BC}(X')$ is very general, then $F=\R^+[C]$ is a $(K_{X'}+B'+\bbeta _{X'}+t\omega ')$-trivial ray in $\overline{\rm NA}(X')$ (see \cite[Lemma 7.1]{DHP24}).

  Suppose that $\R^+[C]$ is not an extremal ray in $\overline {\rm NA}(X')$. Then, since $B'+\bbeta _{X'}+t\omega '$ is big, it follows easily from the cone theorem (see Corollary \ref{c-conetheorem}) that we may write $[C]=\sum_{i=1}^k \sigma _i[\Sigma _i]+\gamma$ where $\sigma _i>0$, $\Sigma _i$ are curves spanning extremal rays in $\overline {\rm NA}(X')$ such that $(K_{X'}+B'+\bbeta _{X'}+t\omega ')\cdot \Sigma _i\leq 0$, and either $\gamma =0$ or 
 $\gamma \in \overline {\rm NA}(X')_{K_{X'}+B'+\bbeta _{X'}+t\omega ' >0}$. In particular $k\geq 1$. If $\nu _*\Sigma _i\ne 0$, then $0=[\nu _*C]=\sum \sigma _i\nu _*[\Sigma _i]+\nu _*\gamma$ and so \[ -[\nu _*\Sigma _i]=(\sum _{j\ne i}\frac {\sigma _j}{\sigma _i}\nu _*[\Sigma _j]+\frac {1}{\sigma _i}\nu_*\gamma)\in \overline {\rm NA}(X)\] and we are in Case (1) of the theorem.
Therefore, we may assume that $\nu _*\Sigma _i=0$ for all $1\leq i\leq k$.

 Since $K_{X'}+B'+\bbeta _{X'}+t\omega '$ is nef over $X$, and $\nu _*\Sigma_i =0$, then $(K_{X'}+B'+\bbeta _{X'}+t\omega ')\cdot \Sigma_i =0$ for all $i$. It then follows that $\R^+[C]=\R^+[\Sigma _i]$  by our choice of $\omega '$ and hence $\R^+[C]$ is an extremal ray in $\overline {\rm NA}(X')$, contradicting our assumption.

Suppose $\R^+[C]$ is an extremal ray in $\overline {\rm NA}(X')$. Furthermore we can assume that any irreducible $\Sigma \in \R^+ [C]$ is $\nu$-vertical (hence also $\psi$-vertical by the construction of $\bar X$), otherwise we are in Case (1) again by the above proof.   %There are two contractions: (1) $\psi:X'\to\bar{X}$, which contracts all $\nu$-vertical curves in $F$, and (2) the K\"ahler contraction $c_F:X'\to X''$, which contracts all curves in $F$. A priori, these are not the same. We need to consider the case where some $[\Sigma]=\lambda[C]$ but $\nu_*\Sigma\ne 0$; then $[\nu_*\Sigma]=\lambda[\nu_*C]=0$, so $-[\nu_*\Sigma]=0\in\overline{\mathrm{NA}}(X)$. If all the curves are $\nu$-vertical, then $c_R$ and $\psi$ coincide.}. 
We note that since $K_{X'}+B'+\bbeta _{X'}+s\omega '$ is K\"ahler over $X$ for $s>t$, then $F$ is $(K_{X'}+B'+\bbeta _{X'})$-negative and $-(K_{X'}+B'+\bbeta _{X'})$ is K\"ahler over $\bar X$. 
 It follows that $\psi$ is a flipping or divisorial contraction,  and hence $\bar X$ is K\"ahler (see Theorem \ref{t-contK}) and in fact $K_{\bar X}+\bar B+\bbeta _{\bar  X}+t\bar \omega =\psi _*(K_{X'}+B'+\bbeta _{X'}+t\omega ')$ represents a K\"ahler class.
 Let $X^+\to \bar X$ be the log canonical model of $K_{X'}+B'+\bbeta _{X'}$ over $\bar X$ (see Proposition \ref{c-relmmp}). Note that if $X'\to \bar X$ is a divisorial contraction, then $X^+=\bar X$ and if $X'\to \bar X$ is a flipping contraction, then $X^+\to \bar X$ is the flip. It is easy to see that $K_{X^+}+B^++\bbeta _{X^+}+s\omega ^+$ is K\"ahler over $X$ for any $0<t-s\ll 1$. We may replace $(X',B'+\bbeta _{X'}+s\omega ')$ by $(X^+,B^++\bbeta _{X^+}+s\omega ^+)$ and continue the process. Proceeding in this way we obtain a sequence of distinct log canonical models for $K_{X'}+B'+\bbeta _{X'}+t_i\omega '$ over $X$, where $t_0=t>t_1>t_2>\cdots$.
 The termination of this sequence can be checked locally over $X$ and hence follows by finiteness of models (see the proof of  Proposition \ref{c-relmmp}). Therefore, after finitely many steps, say $n$ steps, we reach $t_n=0$, and $X_n\to X$ is small as explained above.  By Theorem \ref{t-3} there exists a small strongly $\Q$-factorial modification $\mu: X^{\rm qf}\to X_n$ where $\mu$ is projective, and hence $X^{\rm qf}$ is K\"ahler and $X^{\rm qf}\to X$ is also a small strongly $\Q$-factorial modification.
\end{proof}

%\begin{proof} Suppose that $-[C]\in \overline{\rm NA}(X)$, then $\alpha \cdot C\leq 0$ which is impossible.  By Theorem \ref{c-1}, there exists a small $\Q$-factorial K\"ahler modification $\nu :X'\to X$ \end{proof}

\begin{remark}\label{r-1} 
Jia Jia and Sheng Meng \cite{JM25} show that there are many non-K\"ahler Fujiki class $\mathcal C$ manifolds with $h^{1,1}=1$. Let $Y$ be a generic smooth quintic threefold, $H$ a hyperplane section of $Y$, and $C\subset Y$ a smooth rational curve of degree $d$ with normal bundle $\OO(-1)\oplus\OO(-1)$. Let $g:X'\to Y$ be the blow up of $C$, with exceptional divisor $E\cong\mathbb P^1\times\mathbb P^1$. We write $C'_d$ for a fiber of $E\to C$, and $C_d$ for a curve of the other ruling, which $g$ maps isomorphically onto $C$. Let $f:X'\to X$ be the contraction of $E$ along the ruling $C_d$. Then $X$ is a smooth Moishezon threefold that is not projective, and hence not K\"ahler.
\begin{center}
\begin{tikzcd}
    {X'} & Y \\
    X
    \arrow["g", from=1-1, to=1-2]
    \arrow["f"', from=1-1, to=2-1]
\end{tikzcd}
\end{center}
The class $C_d$ is ``extremal" for $X'\to X$. Since $(g^*H+dE)\cdot C_d=0$ and $(g^*H+dE)\cdot C'_d=-d$, then $g^*H+dE=f^*\alpha$ where $\alpha \in H^{1,1}_{\rm{BC}}(X)$ is not nef and every other non-zero class in $H^{1,1}(X)$ is a multiple of $\alpha$ and hence not nef. Since the nef cone of $X$ is $\{0\}$, then $\overline{{\rm NA}}(X)=H^{1,1}_{\rm{BC}}(X)^\vee$. 
On the other hand, since $C'_d\to f(C'_d)$ is an isomorphism, $\alpha \cdot f(C'_d)=-d$. So $g^*H+dE$ defines $X'\to X$ over $X$ but is not semi-ample (absolutely). 
Let $L$ be the strict transform of a line disjoint from $C_d$, then on $X'$ we have $C_d\equiv C'_d+dL$ so that $0=f_*C_d \equiv f_*(C'_d+dL)$ which tells us that $-f_*C'_d\equiv f_*(dL)$ in $\overline{{\rm NA}}(X)$.

Note that $C_d$ is relatively extremal over $X$, 
while $C_{d}'$ is not in the relative Mori cone 
$\overline{\mathrm{NA}}(X'/X)$. 
Both $C_d$ and $C_{d}'$ are $K_{X'}$-negative curves. 
The absolute $K_{X'}$-MMP can contract $C'_d$ but not $C_{d}$, 
whereas the relative MMP over $X$ can only contract $C_d$. 
In this example, we have
$$
\overline{\mathrm{NA}}(Y) = \mathbb{R}_+[L], 
\quad 
\overline{\mathrm{NA}}(X) = H^{1,1}_{\rm{BC}}(X)^{\vee}, 
\quad 
\overline{\mathrm{NA}}(X') = \mathbb{R}_+[L] + \mathbb{R}_+[C_{d}'].
$$
In particular, $C_d$ lies in the interior of the Mori cone 
$\overline{\mathrm{NA}}(X')$ (since $C_d \equiv C_{d}' + dL$). 
The relative Mori cone $\overline{\mathrm{NA}}(X'/X) = \mathbb{R}_+[C_d]$ 
fails to be an extremal face of $\overline{\mathrm{NA}}(X')$.
\end{remark}

\begin{corollary} Let $X$ be a compact complex analytic variety in Fujiki's class $\mathcal C$ such that $(X,B+\bbeta )$ is a generalized klt pair with $B+\bbeta _X$ modified big, then $K_X+B+\bbeta _X$ is K\"ahler if and only if $K_X+B+\bbeta _X$ is nef and $(K_X+B+\bbeta _X)\cdot C>0$ for every curve $C\subset X$.
\end{corollary}
\begin{proof} One direction is immediate. If $K_X+B+\bbeta _X$ is K\"ahler, then $K_X+B+\bbeta _X$ is nef and $(K_X+B+\bbeta _X)\cdot C>0$ for every curve $C\subset X$. 

We now consider the converse implication. Thus we assume that $\alpha :=K_X+B+\bbeta _X$ is nef and that $\alpha \cdot C>0$ for every curve $C\subset X$.

Suppose that there is a curve $C$ with $-[C]\in \overline{\rm NA}(X)$, then, since $\alpha$ is nef, we have $\alpha \cdot C\leq 0$, a contradiction. Thus, if $X$ is not K\"ahler, then by Theorem \ref{c-1} there is a small $\Q$-factorial K\"ahler modification $\mu :X^{\rm qf}\to X$; in particular $X^{\rm qf}$ is K\"ahler. Write $\alpha ^{\rm qf}:=\mu ^*\alpha =K_{X^{\rm qf}}+B^{\rm qf}+\bbeta _{X^{\rm qf}}$. Since $\mu$ is small, $B^{\rm qf}=\mu ^{-1}_*B\geq 0$, the pair $(X^{\rm qf},B^{\rm qf}+\bbeta )$ is generalized klt, and $\alpha ^{\rm qf}$ is nef.

By Lemma \ref{l-kltbig} there is a generalized klt pair $(X,G+\ggamma )$ and a resolution $\nu :X'\to X$, which we may assume factors as $\nu =\mu \circ q$ with $q:X'\to X^{\rm qf}$ proper bimeromorphic, such that $K_X+B+\bbeta _X\equiv K_X+G+\ggamma _X$, $\ggamma$ descends to $X'$, and $\ggamma _{X'}-\omega '$ is nef for some K\"ahler form $\omega '$ on $X'$. Fix a K\"ahler form $\omega ^{\rm qf}$ on $X^{\rm qf}$. Since $X'$ is compact, $\omega '-\epsilon q^*\omega ^{\rm qf}$ is K\"ahler for $0<\epsilon \ll 1$, and so, setting $\ddelta :=\ggamma -\epsilon \overline{\omega ^{\rm qf}}$, the class
\[\ddelta _{X'}=(\ggamma _{X'}-\omega ')+(\omega '-\epsilon q^*\omega ^{\rm qf})\]
is nef, i.e. $\ddelta:=\ggamma -\epsilon \overline{\omega ^{\rm qf}}$ is a nef b-class. Let $G^{\rm qf}=\mu ^{-1}_*G$. Note that the pair $(X^{\rm qf},G^{\rm qf}+\ddelta )$ is also generalized klt, and
\[\alpha ^{\rm qf}\equiv K_{X^{\rm qf}}+G^{\rm qf}+\ddelta _{X^{\rm qf}}+\epsilon \omega ^{\rm qf}.\]
By \cite[Lemma 2.45]{HX26}, $\alpha ^{\rm qf}$ is NQC, and hence $\alpha$ is NQC by Lemma \ref{l-nefpb} and the proof of \cite[Lemma 2.44]{HX26}.

If $\alpha$ is not big, then by Theorem \ref{t-raynonbig}, $X$ is covered by $\alpha$-trivial rational curves, which contradicts $\alpha \cdot C>0$. If $\alpha$ is nef and big but not K\"ahler, then by Theorem \ref{t-ray} there is an $\alpha$-trivial rational curve, again a contradiction. Therefore $\alpha$ is K\"ahler.

\end{proof}

\bigskip

\noindent {\bf Data availability statement.} Data sharing not applicable to this article as no datasets were generated or analysed during the current study.

\noindent {\bf Conflict of interest statement.} On behalf of all authors, the corresponding author states that there is no conflict of interest.

\end{document}